\documentclass[a4paper, 11pt]{article}

\usepackage{fullpage} 
\usepackage{parskip} 
\usepackage[numbers]{natbib}
\usepackage[utf8]{inputenc}
\usepackage[T1]{fontenc}
\usepackage{lmodern}
\usepackage{enumitem}
\usepackage[center]{caption}
\usepackage{graphicx}
\usepackage{placeins}
\usepackage{amssymb}
\usepackage{hyperref}
\usepackage{bm}
\usepackage{physics}
\usepackage{tabularx}
\usepackage{booktabs}
\usepackage{mathrsfs}
\usepackage[textfont={small,sl},labelfont={small,bf,sl}]{subcaption}
\usepackage{ragged2e}
\usepackage{tikz}
\usepackage{amsthm}
\usepackage{enumitem}
\newtheorem{theorem}{Theorem}[section]
\newtheorem{prop}{Proposition}[section]

\newtheorem{lemma}[theorem]{Lemma}
\newtheorem{remark}{Remark}[section]
\newtheorem{definition}{Definition}[section]
\usepackage{amsmath} 
\usepackage{amsfonts} 
\usepackage{float}

\hypersetup{colorlinks=true, 
breaklinks= true, 
urlcolor= magenta, 
linkcolor= blue, 
citecolor= blue, 
pdftitle= {}, 
pdfauthor= {Imane ESSADEQ}, 
pdfsubject= {},
pdfborder= {0 0 0}}

\newcommand{\poubelle}[1]{}
\numberwithin{equation}{section}
\usepackage{authblk}

\title{Stability and Convergence of Modal Approximations in Coupled Thermoelastic Systems: Theory and Simulation}

\author[1]{I. Essadeq}
\author[1]{S. Nafiri}
\author[2]{S. Benjelloun}
\author[1]{A. E. Fettouh}

\affil[1]{\footnotesize LaGeS Laboratory, Modeling and Numerics Team (MoNum), Hassania School of Public Works, Casablanca, Morocco. Emails: essadeq.imane@gmail.com, nafiri@ehtp.ac.ma and anoirfetto@hotmail.com}
\affil[2]{\footnotesize DeVinci Higher Education, De Vinci Research Center, Paris, France, Email: saad.benjelloun@devinci.fr.}
\date{\textit{Dedicated to the memory of Professor Hammadi Bouslous}}

\begin{document}

\maketitle


\begin{abstract}

In this work, we review and analyze both the theoretical and numerical aspects of strongly and weakly coupled thermoelastic systems. By employing spectral analysis techniques and establishing uniform resolvent estimates, we derive uniform polynomial decay rates for the associated semigroups under a suitable class of boundary conditions. Particular attention is paid to the role of modal approximations in energy analysis. The theoretical results are complemented by numerical experiments that illustrate how the regularity of initial data, smooth versus nonsmooth, affects the observed decay rates, providing deeper insight into the interplay between spectral structure and energy dissipation.
\end{abstract}

{\bf Key words:}
{Semigroup theory; Spectral analysis; Modal approximation; Thermoelasticity; Exponential and polynomial stability; Resolvent estimates.}

{\bf MSC Classification (2020):}
{35B40, 35L55, 35M33, 35K05, 65M60, 65M12, 35Q74, 93D20}

\tableofcontents

\section{Introduction}
\subsection{Background and Motivation}
Thermoelasticity lies at the intersection of mechanics and heat transfer, with broad applications in civil and structural engineering, materials science, and fluid mechanics. It models the interaction between mechanical deformation and temperature variations (interactions) that play a critical role in damping, wave propagation, and material stability; see \citet{gurtin2010}.

The first known work on thermoelasticity was presented by Duhamel before the French Academy of Sciences in 1835, and later published in the Journal de l'Ecole Polytechnique in 1837 \citet{duhamel1837second}. In this seminal paper, Duhamel formulated boundary value problems and derived the fundamental equations describing the coupling between the temperature field and the mechanical deformation of a solid body.

Real-world applications vividly illustrate the significance of thermoelastic coupling. In railway engineering, for example, thermal expansion and contraction of rails must be carefully managed to maintain track alignment and ensure safety under extreme weather conditions \citep{mirkovic2023parametric,yang2016thermal,zhai2019train}. Similar challenges arise in bridge expansion joints \citep{li2023thermal} and telecommunication cables \citep{kotsovinos2020analysis}, where mechanical and thermal interactions demand accurate modeling and control. A notable example is the Hubble Space Telescope, where thermally induced vibrations impaired performance, underscoring the importance of modeling coupled thermal and mechanical effects in flexible structures \citep[p.~1163]{gibson1992approximation}. For comprehensive treatments of energy transfer and damping in such systems, we refer to \citet{carlson1973linear}, \citet{chadwick1960thermoelasticity}, and \citet{gibson1992approximation}.

The stability analysis of thermoelastic systems is crucial because of their widespread engineering applications. In civil structures such as bridges, excessive or uncontrolled damping can compromise integrity, making weakly coupled thermoelastic models \eqref{wc_thrm} more appropriate to capture slow energy dissipation and allow effective control. In contrast, in railway systems, fast and controlled thermal contraction is often beneficial to maintaining operational safety, favoring the use of strongly coupled thermoelastic models \eqref{sc_thrm}. These contrasting requirements underscore the need for a rigorous understanding of thermoelastic dynamics to design systems that are both resilient and responsive to coupled phenomena.

To manage such challenges, simplified but effective mathematical models are essential. In many cases, the thermoelastic system is reduced to a one-dimensional model, justified by the dominance of wave and heat propagation along a principal axis. These reduced models retain the essential physical behavior while allowing for tractable analysis and efficient numerical implementation.

Despite the extensive literature on thermoelastic systems, a fundamental difficulty persists when passing from continuous models to numerical approximations, particularly in relation to long-time stability and energy decay. While the continuous dynamics of thermoelastic systems are now well understood under various coupling mechanisms, much less is known about whether discrete approximations faithfully reproduce these decay properties uniformly with respect to the discretization parameters. This issue is especially pronounced for weakly coupled thermoelastic systems, where dissipation mechanisms are indirect and decay rates are typically polynomial rather than exponential. Understanding how such subtle asymptotic behaviors are affected by approximation schemes is therefore both mathematically challenging and practically relevant, motivating the present study.

\subsection{Mathematical Formulation and State of the Art}
One classical formulation, developed by \citet{day1985} for homogeneous rods, consists of a coupled system involving the wave and heat equations:
\begin{equation*}
\begin{aligned}
\rho \partial^2_t u &= (\lambda + 2 \mu) \partial^2_x u - \alpha(3\lambda + 2 \mu) \partial_x \theta, \\
h_0 \partial_t\theta &= k \partial^2_x\theta - \theta_0 (3\lambda + 2 \mu) \partial_x\partial_t u.
\end{aligned}
\label{init_pb}
\end{equation*}
By making a change of variables.
\begin{align*} 
\theta &\rightarrow \frac{\theta - \theta_{0}}{\theta_{0}}, \quad \quad u \rightarrow \frac{u}{L} \sqrt{\frac{\lambda+2 \mu}{\theta_{0} h_{0}}}, \quad \quad t \rightarrow \frac{k t}{h_{0} L^{2}}.
\end{align*}
The equations are transformed to
\begin{equation*}
\begin{aligned}
\partial^2_t u &= c^{2} \partial^2_x u - c^{2} \gamma \partial_x \theta, \\
\partial_t\theta &= \partial^2_x\theta - \gamma \partial_x\partial_t u.
\end{aligned}
\end{equation*}
The constants \( \gamma \) and \( c \) are given by
\begin{align*}
\gamma &= \left( \frac{\theta_{0} \alpha^{2}(3 \lambda + 2 \mu)^{2}}{h_{0}(\lambda + 2 \mu)} \right)^{1 / 2}, \quad \quad c^{2} = \frac{h_{0}^{2}(\lambda + 2 \mu) L^{2}}{k^{2} \rho}.
\end{align*}

\begin{table}[h!]
\centering
\caption{Notation for coefficients and variables in the thermoelastic system}
\begin{tabularx}{\textwidth}{@{} l X l @{}}
\toprule
\textbf{Symbol} & \textbf{Description} & \textbf{Units} \\
\midrule
$u$        & Displacement field (along the domain)               & meters (m) \\
$\theta$   & Absolute temperature field                          & kelvin (K) \\
$\partial_x u$ & Strain & Dimensionless \\
$\partial_x\partial_t u$ & Rate of change of strain & $1/\text{s}$ \\
$\partial_t u$ & Velocity & m/s \\
$\partial^2_t u$ & Acceleration & m/s$^2$ \\
$\partial_x \theta$ & Temperature gradient (heat flux) & K/m \\
$\partial_t \theta$ & Rate of change of temperature & K/s \\
$\partial^2_x \theta$ & Heat diffusion term (in Fourier’s law) & K/m$^2$ \\
$\theta_0$      & Reference (initial) temperature                     & kelvin (K) \\
$\rho$          & Mass density of the material                        & kg/m$^3$ \\
$h_0$           & Specific heat at constant strain                    & J/(kg$\cdot$K) \\
$k$             & Thermal conductivity                                & W/(m$\cdot$K) \\
$\lambda$       & First Lamé constant                                 & pascal (Pa) \\
$\mu$           & Second Lamé constant (shear modulus)                & pascal (Pa) \\
$\alpha$        & Linear thermal expansion coefficient                & 1/K \\
$L$             & Length of the domain                                & meters (m) \\
$t$             & Time variable                                       & seconds (s)
 \\
$x$             & Spatial variable                                    & meters (m)
 \\
\midrule
\multicolumn{3}{l}{\textbf{Derived quantities}} \\
\midrule
$\displaystyle \gamma = \left( \frac{\theta_0 \alpha^2 (3\lambda + 2\mu)^2}{h_0(\lambda + 2\mu)} \right)^{1/2}$ & Coupling strength in the scaled system 

(very small in comparison to 1) & dimensionless \\
$\displaystyle c^2 = \frac{h_0^2(\lambda + 2\mu)L^2}{k^2\rho}$ & Normalized wave speed squared 

(often scaled to 1) & dimensionless \\
\bottomrule
\end{tabularx}
\label{tab:thermoelastic-notation}
\end{table}
The normalized system obtained emphasizes the coupling parameter $\gamma$ and the wave speed $c$, forming what is commonly known as the strongly coupled thermoelastic system. In the following analysis, we set $c = 1$ without loss of generality, as this assumption does not affect the validity of the proofs or the conclusions drawn from our results.

The normalized form of the strongly coupled thermoelastic system, subject to Dirichlet boundary conditions and prescribed initial data, is given by:
\begin{equation}
\label{sc_thrm}
 \left\{\begin{array}{llll}
\partial^2_{t} u - \partial^2_x u + \gamma \partial_x\theta = 0,\quad &\text{in}\;\Omega\times(0, \infty),\\
\partial_t\theta - \partial^2_x \theta + \gamma \partial_x\partial_t u = 0,\quad &\text{in}\;\Omega\times(0, \infty),\\
u =0=\theta, \quad &\text{in}\;\partial\Omega\times (0,\infty),\\
u(\cdot,0)=u_{0},\; u_{t}(\cdot,0)=u_{1},\;\theta(\cdot,0)=\theta_0, \quad &\text{in} \;\Omega,
\end{array}
\right.
\end{equation}
where $\Omega$ is a bounded smooth domain of $\mathbb{R}$ and $\partial\Omega$ its boundary. Here, the coupling occurs through spatial and mixed derivatives, reflecting a strong interaction between temperature and displacement. A weakly coupled variant, introduced by \citet{khodja1997dynamical}, see also \citet{liu2005characterization}, replaces the spatially differentiated terms by direct coupling:
\begin{equation}
\label{wc_thrm}
 \left\{\begin{array}{llll}
\partial^2_{t} u -  \partial^2_{x} u +  \gamma \theta = 0, \quad & \text{in} \;\Omega\times(0, \infty),\\
\partial_t\theta - \partial^2_{x}\theta - \gamma \partial_t u = 0,\quad &\text{in} \;\Omega\times(0, \infty),\\
u=0=\theta,\quad &\text{in}\;\partial\Omega\times (0,\infty),\\
u(\cdot,0)=u_{0},\; u_{t}(\cdot,0)=u_{1},\;\theta(\cdot,0)=\theta_0, \quad & \text{in} \;\Omega.
\end{array}
\right.
\end{equation}
The energy functional associated with the systems \eqref{sc_thrm} and \eqref{wc_thrm} is defined by:
\begin{equation}
\label{thermo_energy}
E(t):=\frac{1}{2}\int_\Omega\Big(|\partial_t u(x,t)|^2+|\partial_x u(x,t)|^2+|\theta(x,t)|^2\Big)dx.
\end{equation}
Moreover, the dissipation of the energy \eqref{thermo_energy} is given by
\begin{equation}
\label{dissip_energy}
\frac{d}{dt}E(t)=-\int_\Omega |\partial_x\theta(x,t)|^2dx.
\end{equation}
From \eqref{dissip_energy}, we see that the total energy \eqref{thermo_energy} decreases due to thermal dissipation; see \citet{dafermos1970asymptotic}. Although both systems \eqref{sc_thrm} and \eqref{wc_thrm} are dissipative, they exhibit fundamentally different asymptotic behaviors. This distinction arises from the nature of their coupling mechanisms: the system \eqref{sc_thrm} is strongly coupled through the terms $\partial_x \theta$ and $\partial_x \partial_t u$, while the system \eqref{wc_thrm} is only weakly coupled via $\theta$ and $\partial_t u$. This structural difference has a profound impact on the analytical properties and long-term dynamics of each system, particularly with respect to energy decay rates and stability characteristics.

\citet{dafermos1968existence} was among the first to investigate the asymptotic behavior of solutions to system \eqref{sc_thrm}, showing that under certain initial conditions, the associated energy decays to zero as time tends to infinity. However, the precise rate of decay remained an open question for more than two decades. \citet{Hansen1992} succeeded in establishing an exponential decay estimate for special boundary conditions, by using Fourier series expansion and decoupling techniques. Substantial progress in the treatment of this (exponential) behavior was achieved through the works of \citet{gibson1992approximation}, \citet{kim1992energy}, \citet{liu1993exponential}, \citet{burns1993energy}, \citet{rivera1992energ}, \citet{fabiano2001renorming}. For a
chronological treatment of the topic, see the book by \citet[Chap 2, p:23-25]{LiuZheng1999}.

In parallel, several research efforts have been interested in the controllability properties of the strongly coupled thermoelastic system \eqref{sc_thrm}. In particular, the works by \citet{zuazua1995controllability}, \citet{lebeau1998null}, and \citet{lebeau1999decay} examined how the differing time scales of the heat and wave components impact the ability to control the system. These studies demonstrated that controllability can be achieved through various mechanisms, including boundary control, internal control forces, and control acting on specific subdomains. Such contributions laid the groundwork for understanding how to steer thermoelastic systems toward desired states despite the inherent coupling between hyperbolic and parabolic dynamics.

In contrast to the strongly coupled thermoelastic system \eqref{sc_thrm}, analyzing the asymptotic behavior of the weakly coupled system \eqref{wc_thrm} presents greater challenges. Although several theoretical works, such as those by \citet{khodja1997dynamical}, \citet{lebeau1999decay}, \citet{ammari2001stabilization}, \citet{alabau2002indirect}, \citet{liu2005characterization}, \citet{batkai2006polynomial}, \citet{burq2006energy}, and \citet{borichev2010optimal}, have established polynomial energy decay under specific conditions, the convergence rate is typically slower and more sensitive to parameters such as coupling strength, regularity of the initial data, and geometric properties of the domain.

To approximate or to control the solutions of infinite-dimensional systems, researchers and engineers often rely on numerical approximation schemes (see, e.g., \citet{richardson1922weather}, \citet{courant1928}, \citet{lions-magenes-v1}, \citet{Glowinski1984}, \citet{courant1994variational}). Although it was once widely assumed that numerical approximations would naturally inherit the decay properties of their continuous counterparts, subsequent studies have shown that this is not generally the case. As demonstrated by \citet{trefethen2022numerical}, \citet{thomee2007galerkin}, and \citet{zuazua2005propagation}, discretization, in time, space, or both, can significantly alter the spectral characteristics of a system, potentially affecting uniform stability. In particular, even when the continuous model exhibits exponential or polynomial energy decay, the discrete version may exhibit spurious growth or stagnation unless the numerical scheme is carefully designed to preserve the qualitative behavior of the original system.

Numerous studies have addressed the question of uniform exponential decay in numerical approximations of the strongly coupled thermoelastic system \eqref{sc_thrm}; see \citet{burns1991approximations,liu1994uniform,gibson1992approximation,fabiano2000,fabiano2001,gervasio2004numerical}, for example, \citet{banks1991exponentially,infante1999boundary,zuazua2005propagation,ramdani2007uniformly} for related results on control and uniform stabilization of wave and beam type equations. In contrast, the issue of uniform polynomial decay in approximation schemes for the weakly coupled system \eqref{wc_thrm} has received comparatively less attention. Notable exceptions include the work of \citet{abdallah2013uniformly,MN2016} and \citet{nafiri2021uniform}, which represent important steps toward understanding the uniform polynomial decay of discrete weakly coupled models.

Taken together, these works reveal a clear dichotomy between strongly and weakly coupled thermoelastic systems. For strongly coupled models, uniform exponential decay is now well understood both at the continuous and discrete levels. In contrast, for weakly coupled systems, although polynomial decay of the continuous energy has been established, the corresponding behavior of numerical or modal approximations remains far less explored. In particular, the question of whether polynomial decay rates can be preserved uniformly with respect to the discretization parameter, and how spectral properties and boundary conditions influence this behavior, has not been fully addressed. This gap becomes even more pronounced when considering the sensitivity of decay rates to the regularity of the initial data.

\subsection{Problem Statement and Contributions}
The aim of this paper is to investigate the uniform decay properties of numerical approximations for thermoelastic systems under both strong and weak coupling mechanisms. While uniform exponential decay for approximations of strongly coupled thermoelastic systems is by now relatively well understood, the preservation of polynomial decay in weakly coupled systems remains largely open. In particular, it is unclear whether modal approximation schemes can capture the correct decay rates uniformly with respect to the discretization parameter, especially in the presence of different boundary conditions and varying regularity of the initial data.

This work contributes to filling that gap through a two-fold effort:
\begin{enumerate}
\item Theoretical Analysis:
\begin{itemize}
\item We revisit and extend existing results on the exponential stability of modal approximations for the strongly coupled thermoelastic system under appropriate boundary conditions. (Proposition \ref{stab_BC}). 
\item We establish new results on polynomial stability for the weakly coupled system, under appropriate boundary conditions and in a suitable modal approximation framework (Theorems \ref{spec_sw}, \ref{theorem_2}, \ref{theorem_3} and Propositions \ref{prop2} and \ref{prop3}).
\item The analysis is carried out using semigroup theory, spectral analysis, and resolvent estimates, ensuring that the decay properties of the continuous model are preserved in the numerical approximations.
\end{itemize}
\item Numerical Validation:
\begin{itemize}
\item We design and implement numerical experiments that illustrate the asymptotic behavior of both systems, with a focus on the influence of initial data regularity and boundary conditions.
\item The simulations demonstrate that, unlike the strongly coupled case, the weakly coupled system exhibits energy decay rates that are highly sensitive to the smoothness or discontinuity of the initial data.
\end{itemize}
\end{enumerate}
The remainder of the paper is structured as follows: Section \ref{section2} introduces the mathematical framework, notation, and approximation strategy. Section \ref{section3} is devoted to the spectral and resolvent-based analysis of the strongly coupled system and establishes its uniform exponential stability under various boundary conditions. Section \ref{section4} presents the main theoretical results on uniform polynomial decay for the weakly coupled system, together with their proofs. Section \ref{section5} provides a series of numerical experiments that validate the theoretical findings and highlight subtle behaviors related to regularity. Finally, Section \ref{section6} concludes the paper and outlines several directions for future research.

\section{Preliminaries and Framework Setting}
\label{section2}

In this section, we establish the functional framework used throughout the paper. We begin by defining the relevant function spaces and introducing the semigroup formulation for the continuous problem. We then present the modal approximation strategy used for discretization and define the associated discrete energy functional. This setting provides the foundation for both the analytical and numerical results that follow.

\subsection{Notations and Functional Spaces}

Let $\Omega \subset \mathbb{R}^N$ be a bounded smooth domain with smooth boundary $\partial\Omega$, and $N\in\mathbb{N}^+$. We use the standard Lebesgue and Sobolev spaces:
\begin{itemize}
\item $L^p(\Omega)$: space of p-integrable real-valued functions with the norm
\[
\|f\|_{L^{p}}=\left(\int_{\Omega}|f(x)|^{p} d x\right)^{\frac{1}{p}},\quad 1\leq p<\infty.
\]
\item $H^{1}(\Omega)$: space of functions in $L^2(\Omega)$ whose first weak derivatives also belong to $L^2(\Omega)$.
\item $H_0^1(\Omega)$: subspace of $H^1(\Omega)$ with functions vanishing on $\partial\Omega$.
\item $H^2(\Omega)$: space of functions whose first and second weak derivatives are square-integrable.
\item $H^m(\Omega)=\Big\{f\in L^2(\Omega):\ D^\alpha f\in L^2(\Omega)\ \text{for all }\alpha\ \text{with }|\alpha|\le m\Big\},$
where derivatives $D^\alpha f$ are understood in the \emph{weak} (distributional) sense.
\end{itemize}
Let $\mathcal{H}$ denote the Hilbert space $L^2(\Omega) \times L^{2}(\Omega) \times L^{2}(\Omega)$, equipped with the norm  
\[
\| z \|_{\mathcal{H}} = \left( \| \nabla u \|^{2}_{L^{2}} + \| v \|^{2}_{L^{2}} + \| \theta \|^{2}_{L^{2}} \right)^{\frac{1}{2}},
\]

for $z=(\nabla u,v, \theta) \in \mathcal{H}$.
\subsection{Semigroup Formulation}

We reformulate the thermoelastic systems \eqref{sc_thrm} and \eqref{wc_thrm} as abstract Cauchy problem of the form:
\begin{equation}
\left\{\begin{aligned}
\frac{dz}{dt} &= \mathcal{G} z, \\
\quad \mathcal{B} z &= g ,\\
\quad z(0) &= z_0,
\end{aligned}\right.
\label{S_g}
\end{equation}
where $\mathcal{G}$ is a closed operator generating a strongly continuous semigroup $T(\cdot)$ on $\mathcal{H}$, and $\mathcal{B}$, $g$, and $z_0$ denote the boundary operator, boundary data, and initial state, respectively.

Noting $z =\begin{pmatrix} \partial_x u, \partial_t u, \theta \end{pmatrix}^T$, for the strongly coupled thermoelastic system \eqref{sc_thrm}, the generator $\mathcal{S}$ is given by:
\begin{equation*}
\begin{aligned}
\mathcal{S} = 
\begin{pmatrix}
0 & \partial_x & 0 \\
\partial_x & 0 & - \gamma \partial_x \\
0 & -\gamma \partial_x & \partial_x^{2}
\end{pmatrix}.
\end{aligned}
\end{equation*}
For the weakly coupled thermoelastic system \eqref{wc_thrm}, the generator $\mathcal{W}$ is:
\begin{equation*}
\mathcal{W} = \begin{pmatrix}
0 & \partial_x & 0 \\
\partial_x & 0 & - \gamma I \\
0 & \gamma I & \partial_x^{2}
\end{pmatrix}.
\end{equation*}

We consider four cases of boundary conditions \textbf{(BCs)}, including:
\begin{itemize}
\item \textbf{Case 1 (Dirichlet--Dirichlet):}$u=\theta=0$ on $\partial\Omega$.
\item \textbf{Case 2 (Dirichlet--Neumann):}  $u=0$, $\partial_x\theta=0$ on $\partial\Omega$.
\item \textbf{Case 3 (Dirichlet--Neumann):}  $\partial_x u=0$, $\theta=0$ on $\partial\Omega$.
\item \textbf{Case 4 (Neumann--Neumann):} $\partial_x u=\partial_x\theta=0$ on $\partial\Omega$. 
\end{itemize}
Well-posedness for these problems follows from standard semigroup theory and results by \citet{Hansen1992}, \citet{rivera1992energ}, \citet{khodja1997dynamical} and \citet{MN2016}.

\subsection{Energy Functional and Stability Definitions}
The total energy associated with either system is defined by:
\[
E(t):=\frac{1}{2}\int_\Omega\Big(|\partial_t u(x,t)|^2+|\partial_x u(x,t)|^2+|\theta(x,t)|^2\Big)dx.
\]
Its dissipation is given by:
\[
\frac{d}{dt}E(t)=-\int_\Omega |\partial_x\theta(x,t)|^2dx\leq 0.
\]
We now recall the definitions of uniform stability for a sequence of $C_{0}$-semigroups $T_n(\cdot)$ with generators $\mathcal{G}_n$ on Hilbert spaces $\mathcal{H}_n$.
\begin{definition}
\label{stability_definitions}
Let $\mathcal{G}_{n}$, $n \in \mathbb{N}$, be a sequence of generators of $C_0$-semigroups of contractions $T_{n}(\cdot)$ with $n \in \mathbb{N}$, on $\mathcal{L}(\mathcal{H}_{n})$.
\begin{itemize}
    \item \textbf{Uniform Strong Stability:} $T_{n}(\cdot)$, $n \in \mathbb{N}$, is said uniformly strongly stable if 
$$ \lim_{t \to \infty} \left\|T_{n}(t)\right\|_{\mathcal{L}(\mathcal{H}_{n})}  = 0 \text { for all } n \in \mathbb{N}.
$$
    \item \textbf{Uniform Exponential Stability:} $T_{n}(\cdot)$, $n \in \mathbb{N}$, is said uniformly exponentially stable with decay rate $\alpha>0$ if there exists a constant $M$ independent of $n$ such that  
$$
\left\|T_{n}(t)\right\|_{\mathcal{L}(\mathcal{H}_{n})} \leqslant M e^{-\alpha t} \quad \text { for all } t>0 \quad \text{and} \quad n \in \mathbb{N} \text {. }
$$
    \item \textbf{Uniform Polynomial Stability:} $T_{n}(\cdot)$, $n \in \mathbb{N}$, is said uniformly polynomialy stable of order $\alpha>0$ (decay rate) if there exists a constant $M$ independent of $n$ such that
$$
\left\|T_{n}(t) \mathcal{G}_{n}^{-1}\right\|_{\mathcal{L}(\mathcal{H}_{n})} \leqslant \frac{M}{t^{\frac{1}{\alpha}}} \quad \text { for all } t>0 \quad \text{and} \quad n \in \mathbb{N}.
$$
\end{itemize}
\end{definition}

\begin{remark}
Proving the existence of constants $M$ and $\alpha$ is generally challenging, as illustrated by Huang’s counterexample, which demonstrates that even when each matrix $\mathcal{G}_n$ (with semigroup $T_n(t) = e^{\mathcal{G}_n t}$) has its spectrum $\sigma(\mathcal{G}_n)$ uniformly bounded away from zero, the desired exponential estimates may still fail, see \cite{huang1985characteristic}.

To our knowledge, Banks, Ito, and Wang in \cite{banks1991exponentially} were among the first to tackle this lack of uniform exponential stability in the context of weakly damped wave equations. Their study \cite{banks1991exponentially}, supported by numerical experiments, revealed that standard finite-difference and finite-element discretizations may not preserve the expected exponential decay of the discretized energy as the mesh is refined, see also \citet{ramdani2007uniformly}. To address this, they recommended the use of mixed finite element methods.

Subsequent strategies have included regularizing via Tychonoff methods (as in Glowinski et al.\cite{glowinski1990numerical}), filtering out high-frequency components (as explored by Infante and Zuazua \cite{infante1999boundary}), and additional approaches developed by Maniar and Nafiri \cite{MN2016,nafiri2021uniform}.

For a comprehensive overview and further references, Zuazua’s review \cite{zuazua2005propagation} on wave propagation, observation, and control via finite-difference schemes offers an excellent resource.
\end{remark}

\subsection{Approximation Strategy}
In order to compute the solution $z(\cdot)$ of the Cauchy problem \eqref{S_g}, one often has to use various numerical approximation schemes. The most common approach for the approximation of \eqref{S_g} is to consider a sequence of finite-dimensional systems
\begin{equation} \label{Cauchy_pb}
\left\{\begin{aligned}
    \frac{d{z}_{n}(t)}{dt} &= \mathcal{G}_{n} z_{n}(t), \quad &t &\geqslant 0, 
    \\
    \mathcal{B}_n {z}_{n}(t) &= g(t),\quad &t &\geqslant 0, 
    \\
    z_{n}(0) &= z_{0 n}, \quad & 
\end{aligned}\right.
\end{equation}
for $n \in \mathbb{N}$, and where $\mathcal{G}_{n}$ generates a sequence of $C_{0}$-semigroups of operators $T_{n}(t):=e^{\mathcal{G}_{n} t}$ on the Hilbert spaces $\mathcal{H}_{n}$ for all $n \in \mathbb{N}$ and $t>0$. $\mathcal{B}_n$, $g$, and $z_{0n}$ denote the discrete boundary operator, boundary data, and discrete initial state, respectively.

Let $\mathcal{H}_{n}\subset\mathcal{H}$ be a finite-dimensional subspace spanned by the basis functions
\[
E_{j} = \begin{pmatrix}
\phi_{j} \\
0 \\
0
\end{pmatrix}, \quad E_{n+j} = \begin{pmatrix}
0 \\
\psi_{j} \\
0
\end{pmatrix}, \quad E_{2 n+j} = \begin{pmatrix}
0 \\
0 \\
\xi_{j}
\end{pmatrix}, \quad j=1, \ldots, n.
\]

Since $H^{2}(\Omega) \cap H_{0}^{1}(\Omega)$ is dense in $L^{2}(\Omega)$, we can choose $\phi_{i} \in H^{2}(\Omega) \cap H_{0}^{1}(\Omega), \psi_{i} \in H_{0}^{1}(\Omega)$ and $\xi_{i} \in H_{0}^{1}(\Omega)$, see \citet{liu1994uniform} and \citet{LiuZheng1999}. The inner product on \( \mathcal{H}_{n} \) is the one induced by the \( \mathcal{H} \)-product. We consider the approximation to the solution of \eqref{S_g} of the form
\begin{equation}\label{z_n}
z_{n} = \sum_{j=1}^{3 n} \tilde{z}_{j}(t) E_{j}(x),
\end{equation}
which is required to satisfy the following variational system
\[
\left( \dot{z}_{n}, E_{j} \right)_{\mathcal{H}} = \left( \mathcal{G}_n z_{n}, E_{j} \right)_{\mathcal{H}}, \quad j=1, \ldots, 3 n.
\]
Then we have
\begin{equation}
\begin{aligned}
M_{n} \dot{\tilde{z}}_{n} &= 
\begin{pmatrix}
M_{n}^{(1)} & & \\
& M_{n}^{(2)} & \\
& & M_{n}^{(3)}
\end{pmatrix} \begin{pmatrix}
\dot{\tilde{z}}_{n}^{(1)} \\
\dot{\tilde{z}}_{n}^{(2)} \\
\dot{\tilde{z}}_{n}^{(3)}
\end{pmatrix}, \\
&= \begin{pmatrix}
0 & \tilde{D}_{n}^{T} & 0 \\
-\tilde{D}_{n} & 0 & -\gamma \tilde{F}_{n} \\
0 & \gamma \tilde{F}_{n}^{T} & -G_{n}
\end{pmatrix} \begin{pmatrix}
\tilde{z}_{n}^{(1)} \\
\tilde{z}_{n}^{(2)} \\
\tilde{z}_{n}^{(3)}
\end{pmatrix} = \tilde{\mathcal{G}}_{n} \tilde{z}_{n},
\end{aligned}
\label{M_n}
\end{equation}
with
\[
\tilde{z}_{n}^{(i)} = \begin{pmatrix}
\tilde{z}_{(i-1) n+1}, \ldots, \tilde{z}_{i n}
\end{pmatrix}^{T}, \quad i=1,2,3 .
\]
By construction, the matrix \( M_{n}^{(i)} \) is symmetric and positive definite. Therefore, there exists a lower triangular matrix \( L_{n}^{(i)} \) such that \( M_{n}^{(i)} = \left( L_{n}^{(i)} \right)^{T} \left( L_{n}^{(i)} \right) \). Let \( L_{n} = \operatorname{diag} \left( L_{n}^{(1)}, L_{n}^{(2)}, L_{n}^{(3)} \right) \) and denote \( L_{n} \tilde{z}_{n} \) by \( \bar{y}_{n} \). 

To obtain an approximate solution \( z_{n} \), we are led to solve the system of ordinary differential equations
\[
\dot{\bar{y}}_{n} = \mathcal{G}_{n} \bar{y}_{n},
\]
with
\begin{equation*}\label{A_n}
\mathcal{G}_{n} = \begin{pmatrix}
0 & \left( L_{1}^{T} \right)^{-1} \tilde{D}_{n}^{T} L_{2}^{-1} & 0 \\
-\left( L_{2}^{T} \right)^{-1} \tilde{D}_{n} L_{1}^{-1} & 0 & -\gamma \left( L_{2}^{T} \right)^{-1} \tilde{F}_{n} L_{3}^{-1} \\
0 & \gamma \left( L_{3}^{T} \right)^{-1} \tilde{F}_{n}^{T} L_{2}^{-1} & -\left( L_{3}^{T} \right)^{-1} G_{n} L_{3}^{-1}
\end{pmatrix}.
\end{equation*}
It is easy to see that
\[
Re \left( \mathcal{G}_{n} \bar{y}_{n}, \bar{y}_{n} \right)_{\mathbb{R}^{3 n}} = -\left( G_{n} L_{3}^{-1} \tilde{z}_{n}^{(3)}, L_{3}^{-1} \tilde{z}_{n}^{(3)} \right)_{\mathbb{R}^{n}} \leq 0,
\]
provided that \( G_{n} \) is semi-positive definite. In that case, \( \mathcal{G}_{n} \) generates a contraction \( C_{0} \)-semigroup \( T_{n}(\cdot) \) on \( \mathcal{H}_n \) (\citet{burns1991approximations}).

Notice that with the previous choice of basis in $\mathcal{H}_{n}$, \eqref{z_n} defines an isomorphism between $\mathcal{H}_{n}$ and $\mathbb{R}^{3 n}$ which is equipped with the usual inner product. We can thus consider the evolution equation in $\mathcal{H}_{n}$, with $\left(\mathcal{G}_{n} z_{n}, z_{n}\right)_{\mathcal{H}_{n}}=\left(\mathcal{G}_{n} \bar{y}_{n}, \bar{y}_{n}\right)_{\mathbb{R}^{3 n}}$ for $z_{n} \in \mathcal{H}_{n}$.

\begin{remark}
In what follows, we refer to the strongly coupled thermoelastic system by ($S_s$) and to the weakly coupled thermoelastic system by ($S_w$).
\end{remark}

\section{Uniform Exponential Stability of the Strongly Coupled Thermoelastic System $(S_s)$}
\label{section3}

In this section, we review key results from the literature on the spectral analysis and exponential stability of the strongly coupled thermoelastic system $(S_s)$, in both its continuous formulation and its discretized (modal approximation) counterpart, subject to various boundary conditions (BCs). These foundational results provide the theoretical basis for our subsequent analysis of the weakly coupled system $(S_w)$.

\subsection{Spectral Analysis of System ($S_s$)}
In this subsection, we review some results, see \citet{Hansen1992}, on the asymptotic behavior of the spectrum of $\mathcal{S}$ subject to the boundary conditions \textbf{(BCs)}. We recall that the eigenvalues of the operator $\mathcal{S}$ asymptotically
fall on two branches: one branch is along the negative horizontal axis in the complex plane and the other branch is asymptotic to a vertical line. The following result is proved in \citet{guo1997asymptotic}. \citet{Hansen1992,khodja1997dynamical} showed a similar result.

\begin{theorem}[\citet{guo1997asymptotic}]
Asymptotically, there are at most two branches for the eigenvalues of operator $\mathcal{S}$. One branch is along the negative axis and the other branch approaches a vertical line $\Re\lambda = -\frac{\gamma^2}{2}$ parallel to the imaginary
axis.
\end{theorem}

\subsection{Exponential Stability of System $(S_s)$}

The asymptotic behavior of solutions to the strongly coupled thermoelastic system has been the subject of extensive investigation. The pioneering work of \citet{dafermos1968existence} established that the energy of the system decays to zero as $t\to\infty$, though no explicit decay rate was given. Subsequent breakthroughs by \citet{Hansen1992}, \citet{liu1993exponential}, and \citet{rivera1992energ} demonstrated that, under suitable boundary conditions, the system generates an exponentially stable semigroup on the corresponding energy space.

More precisely, let $\mathcal{S}$ be the infinitesimal generator of the semigroup $T(t):=e^{t\mathcal{S}}, t>0$, associated with system $(S_s)$. Then, there exist constants $M>0$ and $\alpha>0$ such that
\[
\|T(t)z_0\|_\mathcal{H}\leqslant M e^{-\alpha t}\|z_0\|_\mathcal{H},\quad\text{ for all } t>0,\quad\text{ for all } z_0\in\mathcal{H}.
\]
The exponential stability is typically proven using energy multipliers, compactness-uniqueness arguments, or frequency domain methods.

\subsection{Uniform Exponential Stability of Approximation of $(S_s)$}

The following theorem establishes conditions for the uniform exponential stability of a sequence of \(C_{0}\)-semigroups and their corresponding infinitesimal generators. This theorem extends the well-known frequency domain
approach of exponential stability proved by \citet{gearhart1978spectral,pruss1984spectrum,huang1985characteristic}.
\begin{theorem}[\citet{liu1994uniform}, Corollary 2.3]
\label{thm3}
For $T_{n}(\cdot)$, $n \in \mathbb{N}$, a sequence of contraction semigroups in the Hilbert spaces \( \mathcal{H}_{n} \), $n \in \mathbb{N}$, with \( \mathcal{A}_{n} \) the corresponding infinitesimal generators, the sequence $T_{n}(\cdot)$, $n \in \mathbb{N}$, is uniformly exponentially stable if and only if the following conditions are satisfied
\begin{enumerate}
\item $\sup_{n \in \mathbb{N}} \{ \text{Re} \lambda : \lambda \in \sigma(\mathcal{A}_{n}) \} = \sigma_{0} < 0,$
\item $\sup_{\text{Re} \lambda \geq 0, n \in \mathbb{N}} \{ \| (\lambda I - \mathcal{A}_{n})^{-1} \| \} < \infty.$
\end{enumerate}

\end{theorem}

\subsubsection{Approximation of System $(S_s)$ under Dirichlet--Dirichlet (BCs)}
\label{1_sub_DD}

Considering the Dirichlet--Dirichlet BCs, both ends of the rod are clamped to the reference temperature, and
\[
\left. u \right|_{x=0, \pi} = 0 \quad \text{and} \quad \left. \theta \right|_{x=0, \pi} = 0 \quad \text{for} \quad t > 0.
\]
In their study, \citet{liu1994uniform} demonstrate the uniform exponential stability of system $(S_s)$ using modal approximation scheme. This approximation uses the eigenvectors of the system $(S_s)$ with \( \gamma = 0 \) as basis vectors. 
\[
\phi_{j}=\sqrt{\frac{2}{\pi}} \frac{1}{j} \sin j x, \quad \psi_{j}=\sqrt{\frac{2}{\pi}} \sin j x, \quad \xi_{j}=\sqrt{\frac{2}{\pi}} \sin j x, \quad j=1, \ldots, n.
\]
In \eqref{M_n}, we consider 
\[
\begin{aligned}
& \left( M_{n}^{(1)} \right)_{ij} = \left( \partial_x \phi_{i},  \partial_x \phi_{j} \right)_{L^{2}}, \quad \left( M_{n}^{(2)} \right)_{ij} = \left( \psi_{i}, \psi_{j} \right)_{L^{2}}, \quad \left( M_{n}^{(3)} \right)_{ij} = \left( \xi_{i}, \xi_{j} \right)_{L^{2}}, \\
\end{aligned}
\]
\[
\begin{aligned}
& \left( \tilde{D}_{n} \right)_{ij} = \left(  \partial_x \phi_{i},  \partial_x \psi_{j} \right)_{L^{2}}, \quad \left( \tilde{F}_{n} \right)_{ij} = \left( \partial_x \xi_{i}, \psi_{j} \right)_{L^{2}}, \quad \left( G_{n} \right)_{ij} = \left(  \partial_x \xi_{i},  \partial_x \xi_{j} \right)_{L^{2}}.
\end{aligned}
\]
As a result, we obtain the following representation of the matrix $\mathcal{S}_n$
\[
\mathcal{S}_{n}=\left[\begin{array}{ccc}
0 & D_{n} & 0 \\
-D_{n} & 0 & -\gamma F_{n} \\
0 & \gamma F_{n}^{T} & -D_{n}^{2}
\end{array}\right],
\]
where \( D_{n} \) and \( F_n \) are defined as 
$$
D_{n}=\left[\begin{array}{lll}
1 & & \\
& \ddots & \\
& & n
\end{array}\right], \quad F_n = (F_{ij})_{1\leqslant i,j\leqslant n},\text{ with } F_{ij}= \begin{cases}-\frac{4}{\pi} \frac{i j}{i^{2}-j^{2}}, & |i-j|\text { is odd } \\
0, & \text { otherwise. }\end{cases}
$$
\begin{theorem}[\citet{liu1994uniform}, Theorem 3.1]
\label{thm_LZb94}
The semigroups generated by \( \mathcal{S}_{n} \) are uniformly exponentially stable.
\end{theorem}
\citet{liu1994uniform} provided significant insights into the uniform stability of approximations of thermoelastic systems, as formalized in Theorem~\ref{thm_LZb94}. Using a contradiction argument, they showed that all eigenvalues of the operator \( \mathcal{S}_n \) have strictly negative real parts for every $n\in\mathbb{N}$, and that the resolvent family satisfies the uniform bound
\[
\underset{\Re \lambda \geq 0, n \in \mathbb{N}}{\sup} \{ \| ( \lambda I-\mathcal{S}_{n})^{-1} \| \} < \infty.
\]
By Theorem~\ref{thm3}, this resolvent estimate guarantees the uniform exponential stability of the discrete scheme associated to system $(S_s)$.

To further reinforce these results, the convergence of the approximate semigroups is established in the following theorem.
\begin{theorem}[\citet{liu1994uniform} Theorem 3.2]
Let \( T_{n}(\cdot) \) be the semigroup generated by $\mathcal{S}_{n}$. The convergence of \( T_{n}(\cdot) \), $n \in \mathbb{N}$, to \( T(\cdot) \) in \( \mathcal{H} \), is uniform over bounded $t$-intervals.
\end{theorem}

\subsubsection{Approximation of System $(S_s)$ under Dirichlet--Neumann (BCs)}
\label{1_sub_DN}

Considering the Dirichlet--Neumann \textbf{(BCs)}, the rod is clamped, and heat insulation is applied. The boundary conditions at both ends are given by
\begin{equation*}
u|_{x=0,\pi} = 0 \quad \text{and} \quad  \partial_x\theta|_{x=0,\pi} = 0 \quad \text{for} \quad t > 0.
\end{equation*}

Under Dirichlet-Neumann boundary conditions, we consider the following basis vectors
\[
\phi_{j}=\sqrt{\frac{2}{\pi}} \frac{1}{j} \sin j x, \quad \psi_{j}=\sqrt{\frac{2}{\pi}} \sin j x, \quad \xi_{j}=\sqrt{\frac{2}{\pi}} \cos j x, \quad j=1, \ldots, n.
\]
This leads to the following representation of the matrix $\mathcal{S}_{n}$
\[
\mathcal{S}_{n}=\left[\begin{array}{ccc}
0 & D_{n} & 0 \\
-D_{n} & 0 & \gamma D_{n} \\
0 & -\gamma D_{n} & -D_{n}^{2}
\end{array}\right],
\]
with 
$$
D_{n}=\left[\begin{array}{lll}
1 & & \\
& \ddots & \\
& & n
\end{array}\right].
$$
By integrating the heat equation in system ($S_s$) with respect to $x$ and $t$, where $\theta_{0}(t)$ represents the initial temperature distribution of the rod, we obtain the equation
\[
\int_{0}^{\pi} \theta(x, t) \, dx = \int_{0}^{\pi} \theta_{0}(x) \, dx.
\]
Introducing the transformed dependent variable
\[
\bar{\theta} = \theta - \frac{1}{\pi} \int_{0}^{\pi} \theta_{0}(x) \, dx,
\]
and using a similar approach as demonstrated in Section \ref{1_sub_DD}, we can show that system $(S_s)$ remains uniformly exponentially stable under Neumann--Dirichlet boundary conditions. For further insights, we refer to \citet{liu1994uniform}.

\subsubsection{Approximation of system $(S_s)$ under Neumann--Dirichlet (BCs)}

Under Neumann--Dirichlet \textbf{(BCs)}, both ends of the rod are stress free, and fixed to the reference temperature i.e.
\begin{equation*}
\partial_x u|_{x=0,\pi} = 0\quad \text{and} \quad  \theta|_{x=0,\pi} = 0 \quad \text{for} \quad t > 0.
\end{equation*}
We use the following basis vectors
\[
\phi_{j}=\sqrt{\frac{2}{\pi}} \frac{1}{j} \cos j x, \quad \psi_{j}=\sqrt{\frac{2}{\pi}} \sin j x, \quad \xi_{j}=\sqrt{\frac{2}{\pi}} \sin j x, \quad j=1, \ldots, n.
\]
This leads to the following representation of the matrix $\mathcal{S}_{n}$
\[
\mathcal{S}_{n}=\left[\begin{array}{ccc}
0 & D_{n}^{T} & 0 \\
-D_{n} & 0 & -\gamma F_{n} \\
0 & \gamma F_{n}^{T} & -G_{n}
\end{array}\right],
\]
where \( D_{n} \), \( F_{n} \) and \( G_{n} \) are defined as
$$
D_{n}= (D_{ij})_{1\leqslant i,j\leqslant n},\text{ with } D_{ij} = \begin{cases}-\frac{4}{\pi} \frac{i j}{i^{2}-j^{2}}, & |i-j| \text { is even } \\
0 & \text { otherwise }\end{cases},$$
$$
F_n = (F_{ij})_{1\leqslant i,j\leqslant n},\text{ with } F_{ij} = \begin{cases}-\frac{4}{\pi} \frac{i j}{i^{2}-j^{2}}, & |i-j|\text { is odd } \\
0, & \text { otherwise }\end{cases},
$$ 
\[
G_n = \frac{2}{\pi} \left[\begin{array}{lll}
1 & & \\
& \ddots & \\
& & n^2
\end{array}\right].
\]
Following the approach of Section \ref{1_sub_DD}, we can show that system $(S_s)$ remains uniformly exponentially stable. For further insights, we refer to \citet{liu1994uniform}. 

\subsubsection{Approximation of System $(S_s)$ under Neumann--Neumann (BCs)}

Under Neumann--Neumann \textbf{(BCs)}, both ends of the rod are stress free, and heat insulation is applied i.e.
\begin{equation*}
\partial_x u|_{x=0,\pi} = 0 \quad \text{and} \quad  \partial_x\theta|_{x=0,\pi} = 0 \quad \text{for} \quad t > 0.
\end{equation*}
We use the following basis vectors
\[
\phi_{j}=\sqrt{\frac{2}{\pi}} \frac{1}{j} \cos j x, \quad \psi_{j}=\sqrt{\frac{2}{\pi}} \sin j x, \quad \xi_{j}=\sqrt{\frac{2}{\pi}} \cos j x, \quad j=1, \ldots, n.
\]
This leads to the following representation of the matrix $\mathcal{S}_{n}$
\[
\mathcal{S}_{n}=\left[\begin{array}{ccc}
0 & D_{n}^{T} & 0 \\
-D_{n} & 0 & -\gamma F_{n} \\
0 & \gamma F_{n}^{T} & -G_{n}
\end{array}\right],
\]
where \( D_{n} \), \( F_{n} \) and \(G_{n} \) are defined as
$$
D_{n} = (D_{ij})_{1\leqslant i,j\leqslant n},\text{ with } D_{ij} = \begin{cases}-\frac{4}{\pi} \frac{i j}{i^{2}-j^{2}}, & |i-j| \text { is even } \\
0 & \text { otherwise }\end{cases}, \quad F_n=\left[\begin{array}{lll}
1 & & \\
& \ddots & \\
& & n
\end{array}\right],
$$
$$
G_n = \frac{2}{\pi} \left[\begin{array}{lll}
1 & & \\
& \ddots & \\
& & n
\end{array}\right].
$$
Following the approach of Section \ref{1_sub_DD}, we can show that system $(S_s)$ remains uniformly exponentially stable. For further insights, we refer to \citet{liu1994uniform}. 

Following the same arguments of the proof of Theorem \ref{thm_LZb94} in \citet{liu1994uniform}, one can show the following result. These arguments will be used again in the sections that follow.
\begin{prop}
\label{stab_BC}
\begin{itemize}
    \item The semigroups generated by \( \mathcal{S}_{n} \) under Dirichlet-Neumann \textbf{(BCs)} are uniformly exponentially stable.
    \item The semigroups generated by \( \mathcal{S}_{n} \) under Neumann-Dirichlet \textbf{(BCs)} are uniformly exponentially stable.
    \item The semigroups generated by \( \mathcal{S}_{n} \) under Neumann-Neumann \textbf{(BCs)} are uniformly exponentially stable.
\end{itemize}
\end{prop}

\section{Uniform Polynomial Stability of the Weakly Coupled Thermoelastic System $(S_w)$}
\label{section4}
In this section, we study the long-time behavior of the weakly coupled thermoelastic system $(S_w)$ and establish that, unlike the strongly coupled case, its solutions exhibit polynomial decay rather than exponential. We focus on the modal approximation $\mathcal{W}_n$ of the system and show that the corresponding semigroups ${T_n(\cdot)}$ satisfy uniform polynomial stability estimates under various boundary conditions \textbf{(BCs)}. Our analysis is based on resolvent bounds and the application of a polynomial stability criterion for sequences of semigroups.

\subsection{Spectral Analysis  and Lack of Exponential Stability for System ($S_w$)}
\label{sectanoiar}
In this subsection, we investigate the asymptotic behavior of the spectrum of $\mathcal{W}$ subject to the boundary conditions \textbf{(BCs)} and we show that the eigenvalues of $\mathcal{W}$ asymptotically fall into two branches: one branch along the negative reals and the other branch is along the imaginary axis in the complex plane. 

\subsubsection{Spectral Problem Formulation}
It is easy to see that $\mathcal{W}^{-1}$ is compact, hence its spectrum consists of eigenvalues only. Thus $\lambda\in\sigma(\mathcal{W})$ if and only if there exists $(\phi,\psi)\neq (0,0)$ such that
\begin{equation}
\label{sys1}
\left\{\begin{array}{l}
\phi^{\prime}=\lambda \phi, \\
\phi^{\prime \prime}-\gamma \psi=\lambda \phi^{\prime}, \\
\gamma \phi^{\prime}+\psi^{\prime \prime}=\lambda \psi.
\end{array}\right.
\end{equation}
To eliminate $\psi$, we differentiate  the second equation of \eqref{sys1} twice and use the third to obtain a fourth-order equation in $\phi$
\begin{equation}
\phi^{(4)}-\gamma \psi^{\prime \prime}=\lambda^{^2} \phi^{\prime \prime}.
\end{equation}
Substituting $\psi^{\prime \prime}=\lambda\psi-\gamma\phi'$ and using the second equation again to eliminate $\psi$, we arrive at
\begin{equation}
\label{4thode}
\phi^{(4)}-\lambda (\lambda+1)\phi^{\prime \prime}+\lambda (\lambda^{2}+\gamma^{2}) \phi=0.
\end{equation}

\subsubsection{General Solution of the Fourth-Order ODE}

\begin{lemma}
\label{sol_4thode}
The general solution to \eqref{4thode} is
\begin{equation}
\label{sol_ode4}
\phi(x)=\alpha e^{ax}+\beta e^{-ax}+\mu e^{bx}+\delta e^{-bx}
\end{equation}
where $(\alpha,\beta,\mu,\delta)\neq (0,0,0,0)$ and 
\begin{equation}
\label{roots}
\left\{\begin{array}{l}
a=\sqrt{\dfrac{\lambda (\lambda+1)+\sqrt{D}}{2}}, \\
b=\sqrt{\dfrac{\lambda (\lambda+1)-\sqrt{D}}{2}},
\end{array}\right.
\end{equation}
with $D=\big(\lambda(\lambda+1)\big)^{2}-4\lambda\big(\lambda^{2}+\gamma^{2}\big).
$
\end{lemma}
\begin{proof}
Equation \eqref{4thode} has characteristic polynomial
\begin{equation*}
X^{4}-\lambda (\lambda+1) X^{2}+\lambda (\lambda^{2}+\gamma^{2})=0
\end{equation*}
which is biquadratic and has four roots given by $\pm a,\pm b$.
\end{proof}
\begin{remark}
The condition $(\alpha,\beta,\mu,\delta)\neq (0,0,0,0)$ follows from the assumption that the eigenfunction pair $(\phi,\psi)$ is nontrivial, and it will be essential when applying boundary conditions to derive characteristic equations for $\lambda$.
\end{remark}
\subsubsection{Asymptotic Estimates of Parameters}
\begin{lemma}
\label{lem_estimate_ab}
For $|\lambda|\to\infty$, the following estimates hold
\begin{enumerate}
\item[(i)] $a+b=\lambda+O\big(|\lambda|^{1/2}\big)$.
\item[(ii)]  $a-b=\lambda+O\big(|\lambda|^{1/2}\big)$.
\item[(iii)]  $b^{2}=\lambda+O\big(|\lambda|^{-1}\big)$.
\item[(iv)] $a=\lambda+O\big(|\lambda|^{1/2}\big).$
\end{enumerate}
\end{lemma}
\begin{proof}
The proof is divided into three steps.

\textbf{Step 1: Estimates for $a+b$ and $a-b$.}

To establish estimates $(i)$ and $(ii)$, we begin by computing $(a+b)^2$ and $(a-b)^2$. We have
\begin{equation*}
\begin{aligned}
(a+b)^2
&= a^2 + b^2 + 2ab \\
&= \lambda(\lambda+1) + \big(4\lambda(\lambda^2+\gamma^2)\big)^{1/2} \\
&= \lambda(\lambda+1) + 2\lambda^{3/2}\big(1+O(|\lambda|^{-2})\big)^{1/2} \\
&= \lambda(\lambda+1) + 2\lambda^{3/2}\big(1+O(|\lambda|^{-2})\big) \\
&= \lambda^2\big(1+O(|\lambda|^{-1/2})\big).
\end{aligned}
\end{equation*}
A similar computation yields
\[
(a-b)^2 = \lambda^2\big(1+O(|\lambda|^{-1/2})\big).
\]
Taking square roots, we obtain
\begin{equation*}
\begin{aligned}
a+b &= \lambda\big(1+O(|\lambda|^{-1/2})\big)
     = \lambda + O(|\lambda|^{1/2}), \\
a-b &= \lambda\big(1+O(|\lambda|^{-1/2})\big)
     = \lambda + O(|\lambda|^{1/2}).
\end{aligned}
\end{equation*}

\textbf{Step 2: Estimate for $b^2$.}

We now derive the estimate for $b^2$ by expanding the discriminant
\[
D = \big(\lambda(\lambda+1)\big)^2 - 4\lambda(\lambda^2+\gamma^2).
\]
A direct computation gives
\begin{equation*}
\begin{aligned}
D
&= \big(\lambda(\lambda-1)\big)^2 - 4\lambda\gamma^2 \\
&= \big(\lambda(\lambda-1)\big)^2\big(1+O(|\lambda|^{-3})\big).
\end{aligned}
\end{equation*}
Hence,
\begin{equation*}
\begin{aligned}
\sqrt{D}
&= \lambda(\lambda-1)\big(1+O(|\lambda|^{-3})\big) \\
&= \lambda(\lambda-1) + O(|\lambda|^{-1}).
\end{aligned}
\end{equation*}
It follows that
\begin{equation*}
\begin{aligned}
b^2
&= \frac{\lambda(\lambda+1)-\sqrt{D}}{2} \\
&= \lambda + O(|\lambda|^{-1}).
\end{aligned}
\end{equation*}

\textbf{Step 3: Estimate for $a$.}

The estimate for $a$ follows directly from the previously established estimates for
$a+b$ and $a-b$.

This completes the proof.
\end{proof}

\subsubsection{Spectral Asymptotics Theorem}
\begin{theorem}
\label{spec_sw}
For all admissible boundary conditions \textbf{(BCs)}, the spectrum of $\mathcal{W}$ asymptotically lies along the imaginary axis and the negative real axis. More precisely, for large $|\lambda|$, one has:
\begin{center}
$\lambda\sim in$ \hspace*{0.5cm} or \hspace*{0.5cm} $\lambda\sim- n^{2}$	
\end{center}
for some large integer $n$. 
\end{theorem}
\begin{remark}
As $n\to\infty$, the spectrum accumulates near the imaginary axis, indicating the absence of a uniform spectral gap. This asymptotic behavior rules out exponential decay and motivates the study of polynomial stability.
\end{remark}
\begin{proof}
Let $\lambda$ be an eigenvalue of $\mathcal{W}$ and $(\phi,\psi)$ be the associated eigenvectors. By \eqref{sys1} and \eqref{sol_ode4} one has
	\begin{equation}
	\label{sys2}
	\left\{\begin{array}{l}
	 \phi(x)=\alpha e^{ax}+\beta e^{-ax}+\mu e^{bx}+\delta e^{-bx}, \\
	\gamma \psi(x)=\phi^{\prime \prime}(x)-\lambda^{2} \phi(x).
	\end{array}\right.
	\end{equation}
\textbf{Case 1: Dirichlet--Dirichlet (BCs)}\\
In this case we have 
\begin{equation*}
\left\{\begin{array}{l}
\phi(0)=\phi(\pi)=0, \\
\psi(0)=\psi(\pi)=0.
\end{array}\right.
\end{equation*}	
Inserting these conditions in \eqref{sol_ode4} yields the following system
\begin{equation}
\left\lbrace
\begin{aligned}
\alpha +\beta +\mu +\delta	&=0,\\
\alpha e^{a\pi}+\beta e^{-a\pi}+\mu e^{b\pi}+\delta e^{-b\pi}	&=0,\\
\alpha a^{2} +\beta a^{2}+\mu b^{2} +\delta b^{2}	&=0,\\
\alpha a^{2}e^{a\pi}+\beta a^{2}e^{-a\pi}+\mu b^{2}e^{b\pi}+\delta b^{2}e^{-b\pi}	&=0,
\end{aligned}
\right.
\end{equation}
which can be rewritten in matrix form as
\begin{equation}
\label{smatrix1}
\left(\begin{array}{cccc}
1 & 1 & 1 & 1 \\
e^{a\pi} & e^{-a\pi} & e^{b\pi} & e^{-b\pi} \\
a^{2} & a^{2} & b^{2} & b^{2} \\
a^{2} e^{a\pi} & a^{2} e^{-a\pi} & b^{2} e^{b\pi} & b^{2} e^{-b\pi}
\end{array}\right)\left(\begin{array}{c}
\alpha \\
\beta \\
\mu \\
\delta
\end{array}\right)=\left(\begin{array}{c}
0 \\
0 \\
0 \\
0
\end{array}\right).
\end{equation}
Since $(\alpha,\beta,\mu,\delta)\neq (0,0,0,0)$, the matrix in \eqref{smatrix1} can't be invertible, i.e. its determinant is zero
\begin{equation*}
\begin{vmatrix}
1 & 1 & 1 & 1 \\
e^{a\pi} & e^{-a\pi} & e^{b\pi} & e^{-b\pi} \\
a^{2} & a^{2} & b^{2} & b^{2} \\
a^{2} e^{a\pi} & a^{2} e^{-a\pi} & b^{2} e^{b\pi} & b^{2} e^{-b\pi}  \notag
\end{vmatrix}
=0.
\end{equation*}
We compute this determinant to get
\begin{equation}
(a-b)^{2}(a+b)^{2}(e^{a\pi}-e^{-a\pi})(e^{b\pi}-e^{-b\pi})=0.
\end{equation}
It follows 
\begin{equation}
\label{product_ab}
	(a-b)^{2}(a+b)^{2}=0 \hspace*{0.5cm} or \hspace*{0.5cm} (e^{a\pi}-e^{-a\pi})(e^{b\pi}-e^{-b\pi})=0.
\end{equation}
Using \eqref{roots} yields
\begin{equation*}
\label{eq_lambda_i}
a^{2}-b^{2}=0.
\end{equation*}
That is $D=0$. Hence, there exists $\lambda_{1},\lambda_{2},\lambda_{3}$ and $\lambda_{4}$ four roots of $D$.

On the other hand
\begin{equation*}
e^{z\pi}-e^{-z\pi}=0 \Longleftrightarrow z=in.
\end{equation*}
Hence, by \eqref{product_ab}
\begin{equation}
a=in \hspace*{0.5cm} or \hspace*{0.5cm} b=in	
\end{equation}
for some integer $n$. By Lemma \ref{lem_estimate_ab}
\begin{equation*}
\left\{\begin{array}{l}
\lambda\sim b^{2}, \\
\lambda\sim a.
\end{array}\right.
\end{equation*}
It follows that
\begin{equation}
\lambda\sim in \hspace*{0.5cm} or \hspace*{0.5cm} \lambda\sim -n^{2}	
\end{equation}
for some large integer $n$.\\[0.3cm]
\textbf{Case 2: Dirichlet--Neumann (BCs)}\\
The conditions in this case are
\begin{equation*}
\left\{\begin{array}{l}
\phi(0)=\phi(\pi)=0, \\
\psi^{\prime}(0)=\psi^{\prime}(\pi)=0,
\end{array}\right.
\end{equation*}
which give after computation
\begin{equation}
\label{sys3}
\left\lbrace
\begin{aligned}
\alpha +\beta +\mu +\delta	&=0,\\
\alpha e^{a\pi}+\beta e^{-a\pi}+\mu e^{b\pi}+\delta e^{-b\pi}	&=0,\\
\alpha l -\beta l+\mu m -\delta m	&=0,\\
\alpha le^{a\pi}-\beta le^{-a\pi}+\mu me^{b\pi}-\delta me^{-b\pi}	&=0,
\end{aligned}
\right.
\end{equation}
where
\begin{equation*}
\left\lbrace
\begin{aligned}
l &=a\big(a^{2}-\lambda^{2}\big), \\
m &=b\big(b^{2}-\lambda^{2}\big).
\end{aligned}
\right.
\end{equation*}
In matrix form \eqref{sys3} becomes
\begin{equation}
\label{smatrix2} 
\left(\begin{array}{cccc}
1 & 1 & 1 & 1 \\
e^{a\pi} & e^{-a\pi} & e^{b\pi} & e^{-b\pi} \\
l & -l & m & -m \\
l e^{a\pi} & -l e^{-a\pi} & m e^{b\pi} & -m e^{-b\pi}
\end{array}\right)\left(\begin{array}{c}
\alpha \\
\beta \\
\mu \\
\delta
\end{array}\right)=\left(\begin{array}{c}
0 \\
0 \\
0 \\
0
\end{array}\right).
\end{equation}
Similarly to the first case, the determinant of the matrix in \eqref{smatrix2} should be zero, that is
\begin{equation*}
\begin{vmatrix}
1 & 1 & 1 & 1 \\
e^{a\pi} & e^{-a\pi} & e^{b\pi} & e^{-b\pi} \\
l & -l & m & -m \\
l e^{a\pi} & -l e^{-a\pi} & m e^{b\pi} & -m e^{-b\pi}  \notag
\end{vmatrix}
=0.
\end{equation*}
A straight calculation lead to the following equation
\begin{equation}
\label{det_DN}
8lm+(l-m)^{2}\big(e^{(a+b)\pi}+e^{-(a+b)\pi}\big)-(l+m)^{2}\big(e^{(a-b)\pi}+e^{(b-a)\pi}\big)=0.
\end{equation}
By Lemma \ref{lem_estimate_ab}
\begin{equation*}
\left\{\begin{array}{l}
 a+b=\lambda+O\big(|\lambda|^{1/2}\big),\\
a-b=\lambda+O\big(|\lambda|^{1/2}\big).
\end{array}
\right.
\end{equation*} 
Consquentely, \eqref{det_DN} takes the following form
\begin{equation}
\label{f_lm}
8lm=4lm\big(e^{\Lambda}+e^{-\Lambda}\big),
\end{equation}
where $\Lambda=\lambda\pi+O\big(|\lambda\pi|^{1/2}\big)$. A straight calculation shows that
\begin{equation*}
lm=0\Longleftrightarrow \lambda\in\{0,+i\gamma,-i\gamma\}
\end{equation*}
which is impossible, since $i\mathbb{R}\subset\rho(\mathcal{A})$. Hence $lm\neq0$, and by \eqref{f_lm}
\begin{equation*}
e^{\Lambda}+e^{-\Lambda}=2
\end{equation*}
that is
\begin{equation*}
\Lambda=2iq\pi.
\end{equation*}
Hence
$$\lambda+O\big(|\lambda|^{1/2}\big)=(2iq).$$
Finally
$$\lambda\sim in$$
for some large integer $n$.\\[0.3cm]
\textbf{Case 3: Neumann--Dirichlet (BCs)}\\
In this case the conditions write as
\begin{equation*}
\left\{\begin{array}{l}
\phi^{\prime}(0)=\phi^{\prime}(\pi)=0, \\
\psi(0)=\psi(\pi)=0.
\end{array}\right.
\end{equation*}
All calculations done
\begin{equation}
\left(\begin{array}{cccc}
a & -a & b & -b \\
ae^{a\pi} & -ae^{-a\pi} & be^{b\pi} &-be^{-b\pi} \\
l^{*} & l^{*} & m^{*} & m^{*} \\
l^{*} e^{a\pi} & l^{*} e^{-a\pi} & m^{*} e^{b\pi} & m^{*} e^{-b\pi}
\end{array}\right)\left(\begin{array}{c}
\alpha \\
\beta \\
\mu \\
\delta
\end{array}\right)=\left(\begin{array}{c}
0 \\
0 \\
0 \\
0
\end{array}\right)
\end{equation}
where
\begin{equation*}
\left\lbrace
\begin{aligned}
l^{*} &=a^{2}-\lambda^{2}\\
m^{*} &=b^{2}-\lambda^{2}.
\end{aligned}
\right.
\end{equation*}
From the previous case, the determinant is zero, i.e. 
\begin{equation*}
\begin{vmatrix}
a & -a & b & -b \\
ae^{a\pi} & -ae^{-a\pi} & be^{b\pi} & -be^{-b\pi} \\
l^{*} & l^{*} & m^{*} & m^{*} \\
l^{*} e^{a\pi} & l^{*} e^{-a\pi} & m^{*} e^{b\pi} & m^{*} e^{-b\pi}  \notag
\end{vmatrix}=0.
\end{equation*}
Hence
\begin{equation}
\label{f_lstarmstar}
8abl^{*}m^{*}+(bl^{*}-am^{*})^{2}\big(e^{(a+b)\pi}+e^{-(a+b)\pi}\big)-(bl^{*}+am^{*})^{2}\big(e^{(a-b)\pi}+e^{(b-a)\pi}\big)=0.
\end{equation}
Similarly to case 2, we can rewrite \eqref{f_lstarmstar} in the following form
\begin{equation}
8abl^{*}m^{*}=4abl^{*}m^{*}\big(e^{\Lambda}+e^{-\Lambda}\big).
\end{equation}
We can remark that $abl^{*}m^{*}=lm\neq0$, and therefore
\begin{equation*}
e^{\Lambda}+e^{-\Lambda}=2.
\end{equation*}
this equation implies that 
$$\lambda\sim in$$
for some large integer $n$.\\[0.3cm]
\textbf{Case 4: Neumann--Neumann (BCs)}\\
In this case the conditions write as
\begin{equation*}
\left\{\begin{array}{l}
\phi^{\prime}(0)=\phi^{\prime}(\pi)=0, \\
\psi^{\prime}(0)=\psi^{\prime}(\pi)=0,
\end{array}\right.
\end{equation*}
which gives after calculation
\begin{equation}
\left(\begin{array}{cccc}
a & -a & b & -b \\
ae^{a\pi} & -ae^{-a\pi} & be^{b\pi} &-be^{-b\pi} \\
l & -l & m & -m \\
l e^{a\pi} & -l e^{-a\pi} & m e^{b\pi} & -m e^{-b\pi}
\end{array}\right)\left(\begin{array}{c}
\alpha \\
\beta \\
\mu \\
\delta
\end{array}\right)=\left(\begin{array}{c}
0 \\
0 \\
0 \\
0
\end{array}\right)
\end{equation}
with determinant 
\begin{equation*}
\begin{vmatrix}
a & -a & b & -b \\
ae^{a\pi} & -ae^{-a\pi} & be^{b\pi} & -be^{-b\pi} \\
l & -l & m & -m \\
l e^{a\pi} & -l e^{-a\pi} & m e^{b\pi} & -m e^{-b\pi}  \notag
\end{vmatrix}=0.
\end{equation*}
Hence
\begin{equation}
(am-bl)^{2}\big(e^{(a+b)\pi}+e^{-(a+b)\pi}-e^{(a-b)\pi}-e^{(b-a)\pi}\big)=0.
\end{equation}
Furthermore,
\begin{equation*}
am-bl=0 
\end{equation*}
implies
\begin{equation*}
ab\big(a^{2}-\lambda^{2}\big)=ab\big(b^{2}-\lambda^{2}\big)
\end{equation*}
that is
\begin{equation*}
a^{^2}=b^{2}
\end{equation*}
hence $D=0$. Therefore, there exists $\lambda_{1},\lambda_{2},\lambda_{3}$ and $\lambda_{4}$ four roots of $D$.
Moreover,
\begin{equation*}
e^{(a+b)\pi}+e^{-(a+b)\pi}-e^{(a-b)\pi}-e^{(b-a)\pi}=0
\end{equation*}
that is
\begin{equation}
\cosh((a+b)\pi)=\cosh((a-b)\pi).
\end{equation}
Using transformation formulas, the last equation becomes 
$$\sinh(a\pi)\sinh(b\pi)=0.$$
Finally 
\begin{equation}
a=in \hspace*{0.5cm} or \hspace*{0.5cm} b=in	
\end{equation}
for some integer $n$, and we conclude, as in Case 1, that
\begin{equation}
\lambda\sim in \hspace*{0.5cm} Or \hspace*{0.5cm} \lambda\sim -n^{2}	
\end{equation}
for some large integer $n$.
\end{proof}

\subsection{Polynomial Stability of $(S_w)$}
The system $(S_w)$ differs fundamentally from $(S_s)$ in the structure of its coupling terms, which are weaker and do not involve spatial derivatives. According to Theorem \ref{spec_sw}, exponential stability does not hold in general for $(S_w)$, either in the continuous or discrete setting. \citet{khodja1997dynamical,liu2005characterization} showed that the semigroup associated to $(S_w)$ is rather polynomially stable.

More precisely, let $\mathcal{W}$ be the infinitesimal generator of the semigroup $T(t):=e^{t\mathcal{W}}, t>0$, associated with system $(S_w)$. Then, there exist constants $M>0$ and $\alpha>0$ such that
\[
\|T(t)\mathcal{W}^{-1}z_0\|_\mathcal{H}\leqslant \frac{M}{t^{\frac{1}{\alpha}}}\|z_0\|_\mathcal{H},\quad\text{ for all } t>0,\quad\text{ for all } z_0\in\mathcal{H}.
\]
\subsection{Uniform Polynomial Stability of Approximation of $(S_w)$}
The following theorems establish conditions under which a sequence of  $C_0$-semigroups and their corresponding infinitesimal generators exhibit uniform polynomial stability. These results extend the polynomial stability criteria previously established by \citet{batkai2006polynomial,borichev2010optimal} for single semigroups.

\begin{theorem}[\citet{MN2016}, Theorem 3.2]
\label{thm4.4}
Let \( T_{n}(\cdot) \), \( n \in \mathbb{N} \), be a uniformly bounded sequence of \( C_{0} \)-semigroups on the Hilbert spaces \( \mathcal{H}_{n} \), \( n \in \mathbb{N} \), and let \( \mathcal{A}_{n} \) be the corresponding infinitesimal generators, such that for all \( n \in \mathbb{N} \), \( i \mathbb{R} \subset \rho(\mathcal{A}_{n}) \). Then for a fixed \( \alpha > 0 \), the following conditions are equivalent

\begin{enumerate}[label=(\roman*)]
    \item \[
    \sup_{|s| \geq 1, n \in \mathbb{N}} \frac{1}{\|s \|^{\alpha}} \| R(i s, \mathcal{A}_{n}) \| < \infty,
    \]
    \item \[
    \sup_{t \geqslant 0, n \in \mathbb{N}} \| t T_{n}(t) \mathcal{A}_{n}^{-\alpha} \| < \infty,
    \]
    \item \[
    \sup_{t \geqslant 0, n \in \mathbb{N}} \| t^{\frac{1}{\alpha}} T_{n}(t) \mathcal{A}_{n}^{-1} \| < \infty.
    \]
\end{enumerate}
\end{theorem}

\begin{theorem}[\citet{MN2016}, Theorem 3.6]
\label{thm3.6}
Let $T_{n}(\cdot), n \in \mathbb{N}$, be a uniformly bounded sequence of $C_{0}$-semigroups on Hilbert spaces $\mathcal{H}_{n}$, \( n \in \mathbb{N} \), and let $\mathcal{A}_{n}$ be the corresponding infinitesimal generators, such that
$$
\begin{gathered}
i \mathbb{R} \subset \rho\left(\mathcal{A}_{n}\right) \text{,} \\
\sup _{n \in \mathbb{N}}\left\|\mathcal{A}_{n}^{-1}\right\|<\infty \text{,}
\end{gathered}
$$
and
$$
\sup _{\rho \in \mathbb{R}, n \in \mathbb{N}}\left\|R\left(i \rho, \mathcal{A}_{n}\right) \mathcal{A}_{n}^{-\alpha}\right\|<\infty \text{,}
$$
for a constant $\alpha>0$. Then there exist a positive constant $\delta$ independent of $n$ such that $[0, \delta] \subset \rho\left(\mathcal{A}_{n}\right)$ and we have
$$
\sup _{\operatorname{Re} \lambda \geq-\delta, n \in \mathbb{N}}\left\{\frac{|\operatorname{Im} \lambda|^{-\alpha}}{|\operatorname{Re} \lambda|}, \lambda \in \sigma_{p}\left(\mathcal{A}_{n}\right)\right\}<\infty.
$$
\end{theorem}

\begin{remark}
Theorem \ref{thm3.6} shows that, in general, uniform polynomial stability cannot be inferred solely from the spectral criterion stated therein, see \citet{MN2016} for more details.
\end{remark}

\subsubsection{Approximation of $(S_w)$ under Dirichlet--Dirichlet (BCs)}

Considering the Dirichlet--Dirichlet \textbf{(BCs)},
\[
\left. u \right|_{x=0, \pi} = 0 \quad \text{and} \quad \left. \theta \right|_{x=0, \pi} = 0 \quad \text{for} \quad t > 0.
\]
\citet{MN2016} demonstrated the uniform polynomial stability of the modal approximation using the eigenvectors of the uncoupled thermoelastic system (\( \gamma = 0 \)) in $(S_w)$ as the basis vectors.
\[
\phi_{j}=\sqrt{\frac{2}{\pi}} \sin j x, \quad \psi_{j}=\sqrt{\frac{2}{\pi}} \sin j x, \quad \xi_{j}=\sqrt{\frac{2}{\pi}} \sin j x, \quad j=1, \ldots, n.
\]
The corresponding matrices $M_{n}^{(i)}$ defined in \eqref{M_n} are expressed as
$$
\begin{array}{r}
\left(M_{n}^{(1)}\right)_{i j}=\left(\partial_x \phi_{i}, \partial_x \phi_{j}\right)_{L^{2}}, \quad\left(M_{n}^{(2)}\right)_{i j}=\left(\psi_{i}, \psi_{j}\right)_{L^{2}}, \quad\left(M_{n}^{(3)}\right)_{i j}=\left(\xi_{i}, \xi_{j}\right)_{L^{2}},
\end{array}
$$
and 
\[
\begin{aligned}
& \left( \tilde{D}_{n} \right)_{ij} = \left( \partial_x \phi_{i}, \partial_x \psi_{j} \right)_{L^{2}}, \quad \left( \tilde{F}_{n} \right)_{ij} = \left( \xi_{i}, \psi_{j} \right)_{L^{2}}, \quad \left( G_{n} \right)_{ij} = \left( \partial_x \xi_{i}, \partial_x \xi_{j} \right)_{L^{2}}.
\end{aligned}
\]
This leads to the following representation of the matrix $\mathcal{W}_n$
$$
\mathcal{W}_n =\left[\begin{array}{ccc}
0 & D_{n} & 0 \\
-D_{n} & 0 & -\gamma I_{n} \\
0 & \gamma I_{n} & -D_{n}^{2}
\end{array}\right],
$$
with
$$
D_{n}=\left[\begin{array}{ccc}
1 & & \\
& \ddots & \\
& & n
\end{array}\right].
$$
\begin{theorem}[\citet{MN2016}, Theorem 4.10]
The semigroups generated by \( \mathcal{W}_{n} \) are uniformly polynomially stable with order $\alpha=2$ and there exist positive constants \( M \), independent of \( n \), such that
$$
\left\|T_{n}(t) \mathcal{W}_{n}^{-1}\right\|_{\mathcal{H}_{n}} \leqslant \frac{M}{t^{\frac{1}{2}}} \quad \text { for } t>0 \quad \text{and} \quad n \in \mathbb{N}.
$$
\label{thm_4.10}
\end{theorem}

\citet{MN2016} showed that the family of generators $\mathcal{W}_n$ satisfies $i \mathbb{R} \subset \rho\left(\mathcal{W}_n \right)$, $n \in \mathbb{N}$, and $\sup _{n \in \mathbb{N}}|| \mathcal{W}_n^{-1} \|<\infty$ and that the approximate system of $(S_w)$ still decays to zero with an order no less than $\alpha=2$.

Using the Trotter-Kato Theorem, \citet{MN2016} showed the strong convergence of the approximating semigroups $T_{n}(t)$, $n \in \mathbb{N}$, to $T(t)=\mathrm{e}^{t \mathcal{W}}$, $t>0$. 

\subsubsection{Approximation of $(S_w)$ under Dirichlet--Neumann (BCs)}
\label{2_sub_DN}

Under Dirichlet--Neumann \textbf{(BCs)}, 
\begin{equation*}
u|_{x=0,\pi} = 0 \quad \text{and} \quad  \partial_x\theta|_{x=0,\pi} = 0 \quad \text{for} \quad t > 0.
\end{equation*}
and the basis vectors are given by
\[
\phi_{j}=\sqrt{\frac{2}{\pi}} \frac{1}{j} \sin j x, \quad \psi_{j}=\sqrt{\frac{2}{\pi}} \sin j x, \quad \xi_{j}=\sqrt{\frac{2}{\pi}} \cos j x, \quad j=1, \ldots, n.
\]
This leads to the following representation of the matrix $\mathcal{W}_n$
\[
\mathcal{W}_{n}=\left[\begin{array}{ccc}
0 & D_{n} & 0 \\
-D_{n} & 0 & \gamma D_{n} \\
0 & -\gamma D_{n} & -D_{n}^{2}
\end{array}\right],
\]
with
$$
D_{n}=\left[\begin{array}{lll}
1 & & \\
& \ddots & \\
& & n
\end{array}\right].
$$
\begin{lemma}
\label{lemma_1}
The family of generators \( \mathcal{W}_n \) satisfies \( i\mathbb{R} \subset \rho (\mathcal{W}_{n}) \), \( n\in \mathbb{N} \), and \( \underset{n\in \mathbb{N}}{\sup} ~ \norm{\mathcal{W}_n^{-1}} < \infty \).
\end{lemma}

\begin{proof}
To prove this lemma, we proceed by contradiction.

Let's assume that there exists a fixed \( m \in \mathbb{N} \) and \( 0\neq\beta_{m} \in \mathbb{R} \) such that \( i \beta_{m} \) belongs to the spectrum of \( \mathcal{W}_m \). Since \(\mathcal{W}_{m}\) is of finite rank, it is compact,
thus, \( i \beta_{m} \) must be an eigenvalue of $\mathcal{W}_m$. 

Consequently, there exists an eigenvector \( z_{m} \in\mathcal{H}_{m} \), normalized such that \( \norm{z_{m}}_{\mathcal{H}_{m}} = 1 \). Let \( y_{m}^T = (u_{m},v_{m},\theta_{m}) \in \mathbb{R}^{3m} \) denote the coordinate representation of $z_m$ also normalized so that \( \norm{y_{m}}_{\mathbb{R}^{3m}} = 1 \).

This leads to the spectral equations
\begin{equation*}
\begin{aligned}
i \beta_{m} z_{m} - \mathcal{W}_{m} z_{m} &= 0, \\
i \beta_{m} y_{m} - W_{m} y_{m} &= 0.
\end{aligned}
\end{equation*}
Taking the inner product of each equation with \( z_{m}\in\mathcal{H}_m \) and $y_m\in\mathbb{R}^{3m}$, respectively, yields
\begin{equation*}
(\mathcal{W}_m z_m, z_m)_{\mathcal{H}_{m}} = (W_m y_m, y_m)_{\mathbb{R}^{3m}} = - \norm{D_{m} \theta_{m}}^{2}_{\mathbb{R}^{m}} = i\beta_m=0.
\end{equation*}
Since \( \beta_m \neq 0 \) by assumption, this leads to a contradiction unless $D_{m}\theta_{m}=0$. However, by the properties of $D_m$, this implies $\theta=0$, and consequently $z_m=0$, , contradicting the normalization $\norm{z_{m}}_{\mathcal{H}_{m}} = 1$. This completes the proof.

Furthermore, considering \( W_m \) as the matrix representation of \( \mathcal{W}_m \), it suffices to show that for one norm, \( \underset{n\in \mathbb{N}}{\sup} ~ \norm{\mathcal{W}_n^{-1}} < +\infty \). 

For \( W_n^{-1} \), straightforward calculations yield

\begin{equation*}
\norm{W_{n}^{-1}}_\infty = \underset{i}{\max~}  \sum_{j} \abs{w_{ij}} = \gamma^{2} + \gamma +1 < \infty.
\end{equation*}

Here, \( \norm{D_{n}^{-2}}_{\infty} = 1 \), which completes the proof.
\end{proof}


\begin{theorem}
\label{theorem_2}
The semigroups generated by \( \mathcal{W}_{n} \) are uniformly polynomially stable with order $\alpha=2$ and there exist positive constants \( M \), independent of \( n \), such that
$$
\left\|T_{n}(t) \mathcal{W}_{n}^{-1}\right\|_{\mathcal{H}_{n}} \leqslant \frac{M}{t^{\frac{1}{2}}} \quad \text { for } t>0 \quad \text{and} \quad n \in \mathbb{N}.
$$
\end{theorem} 

\begin{proof}
The family \( \mathcal{W}_n \), \( n \in \mathbb{N} \), satisfies the hypothesis of Lemma \ref{lemma_1}. 

To prove the uniform polynomial stability of the approximation scheme, we first need to show the following inequality
\begin{equation}
\sup_{\abs{\beta}\geq1, ~n \in \mathbb{N}} \dfrac{1}{\abs{\beta}^2} \norm{(i\beta I_{3n}-\mathcal{W}_n)^{-1}} < \infty.
\label{4.59}
\end{equation}
Then, the result follows straightforwardly by Theorems \ref{thm4.4} and \ref{thm3.6}.

We proceed by a contradiction argument. Assuming that \eqref{4.59} does not hold true implies the existence of a subsequence \( \mathcal{W}_{n} \) (still noted \( \mathcal{W}_{n} \)), a sequence \( \beta_{m} \in \mathbb{R}^{+} \) with \( \beta_{m} \to \infty \) as \( m \to \infty \), and a sequence \( z_{m} \in \mathcal{H}_{n} \) such that as \( m \to \infty \),

\begin{equation*}
\lim_{m \to \infty} \norm{\beta_m^2(i\beta_m I_{3n} -\mathcal{W}_n)z_m} = 0.
\label{4.60}
\end{equation*}

Denote \( y_{m}^T = (u_{m},v_{m},\theta_{m}) \in \mathbb{R}^{3n} \) the corresponding coordinate vector to \( z_m \). This implies

\begin{equation*}
\lim_{m \to \infty} \norm{\beta_m^2(i\beta_m I_{3n} - {W}_n)y_m} = 0,
\label{4.60bis}
\end{equation*}
which translates to 
\begin{equation}
\norm{\beta_m^2(i\beta_m u_m -D_n v_m)} \to 0,
\label{4.61}
\end{equation}
\begin{equation}
\norm{\beta_m^2(i \beta_m v_m +D_n u_m - \gamma D_n \theta_m)} \to 0,
\label{4.62}
\end{equation}
\begin{equation}
\norm{\beta_m^2(i \beta_m \theta_m + \gamma D_n v_m + D_n^2 \theta_m )} \to 0.
\label{4.63}
\end{equation}
Our objective is to show that \( \norm{y_{m}}_{\mathbb{R}^{3n}} \to 0 \) as \( m \to \infty \), thereby reaching a contradiction given \( \norm{y_{m}}_{\mathbb{R}^{3n}} = 1 \). We have
\begin{equation*}
\abs{\Re(\beta_m^2(i \beta_m I_{3n}-\mathcal{W}_n)y_m,y_m)}  \leq  \norm{\beta_m^2(i \beta_m I_{3n}-\mathcal{W}_n)y_m} \to 0.
\label{4.64}
\end{equation*}
Hence, 
\begin{equation*}
\norm{\beta_m D_n \theta_m}^2 = \Re(\beta_m^2(i \beta_m I_{3n}-\mathcal{W}_n)y_m,y_m) \to 0.
\label{4.65}
\end{equation*}
Since \( D_n \) is invertible with \( \norm{D_n} > 1 \), we have
\begin{equation}
\norm{\beta_m \theta_m}  \leq  \norm{\beta_m D_n \theta_m} \to 0,
\label{4.66}
\end{equation}
which implies
\begin{equation}
\norm{\theta_m} \to 0.
\label{4.67}
\end{equation}
Combining \( \norm{y_{m}}_{\mathbb{R}^{3n}} = 1 \) and \eqref{4.67}, we deduce that
\begin{equation}
\norm{\begin{pmatrix}
u_m\\
v_m
\end{pmatrix}}^2_{\mathbb{R}^{2n}} \to 1.
\label{4.68}
\end{equation}
Next, we aim to demonstrate that \( \norm{v_{m}} \) also converges to zero. We take the inner product of \eqref{4.63} with \( v_m \), which yields
\begin{equation*}
i(\beta_m \theta_m,v_m)+(D_n \theta_m,D_n v_m)+\gamma (D_n v_m, v_m) \to 0.
\label{4.69}
\end{equation*}
From \eqref{4.66} and \eqref{4.68}, we have
\begin{equation*}
(\beta_m \theta_m , v_m) \to 0.
\end{equation*}

The inner product of \eqref{4.61} with \( D_n \theta_m \) and \eqref{4.66}, \eqref{4.68} gives
\begin{equation*}
(D_n \theta_m,D_n v_m) \to 0.
\end{equation*}
This further leads to
\begin{equation*}
\norm{v_m} \to 0.
\label{4.70}
\end{equation*}
Using the difference between the inner product of \eqref{4.61} with \( D_n^{-1} v_m \) and the inner product of \eqref{4.62} with \( D_n^{-1} u_m \), we derive
\begin{equation*}
\norm{u_m}^{2} + \norm{v_m}^{2} - \gamma (\theta_m, u_m) \to 0.
\end{equation*}
Given \( (\theta_m, u_m) \to 0 \) from \eqref{4.67} and \eqref{4.68}, we conclude
\begin{equation*}
\norm{u_m} \to 0.
\label{4.71}
\end{equation*}
Therefore, \( \norm{y_{m}}_{\mathbb{R}^{3n}} \to 0 \) as \( m \to \infty \), leading to the anticipated contradiction.
\end{proof}




\begin{theorem}
\label{theorem_3}
The convergence of \( T_{n}(\cdot) \), $n \in \mathbb{N}$, to \( T(\cdot) \) in \( \mathcal{H} \), is strong and uniform over bounded \( t\)-intervals.
\end{theorem}

\begin{proof}

Our objective here is to prove the strong convergence of the approximating semigroups \( T_n(t) \) to \( T(t) = e^{t\mathcal{W}}, t>0, \) using the Trotter-Kato Theorem as outlined in \citet{Pazy}.

Recall that \( W_n \) is the matrix representation of the operator \( \mathcal{W}_n \) and that \( \mathcal{W} \) and \( \mathcal{W}_n \) are dissipative.

\( \mathcal{D} = \mathcal{D}(\mathcal{W}) \cap (H^4 \times H^3 \times H^4) \) is dense in \( \mathcal{H} \) and given \( (I - \mathcal{W})\mathcal{D}(\mathcal{W}) = \mathcal{H} \), it is also known that \( (I - \mathcal{W})\mathcal{D} \) is dense in \( \mathcal{H} \). 


For \( z \in \mathcal{D} \), we have then
\begin{equation*}
z = \sum_{j=1}^{\infty}
\begin{bmatrix}
a_j \begin{pmatrix} \frac{1}{j}\sin(jx) \\ 0 \\ 0 \end{pmatrix}
+b_j \begin{pmatrix} 0 \\ \sin(jx) \\ 0 \end{pmatrix}
+c_j \begin{pmatrix} 0 \\ 0 \\ \cos(jx) \end{pmatrix}
\end{bmatrix},
\end{equation*}

where \( \{ a_j j^3, b_j j^3, c_j^4 \}_{j \geq 1} \) belongs to the sequence space \( l^2 \). Further, we get

\begin{equation*}
\mathcal{W}z = 
\begin{pmatrix}
\sum_{i=1}^\infty b_i \sin(ix) \\
\sum_{i=1}^\infty (-a_i i - \gamma c_i)\sin(ix) \\
\sum_{i=1}^\infty (\gamma b_i - c_i i^2)\sin(ix)
\end{pmatrix},
\label{eq2}
\end{equation*}

and 

\begin{equation*}
\mathcal{W}_n z = 
\begin{pmatrix}
\sum_{i=1}^n b_i \sin(ix) \\
\sum_{i=1}^n (-a_i i - \gamma c_i)\sin(ix) \\
\sum_{i=1}^n (\gamma b_i - c_i i^2)\sin(ix)
\end{pmatrix}.
\label{eq3}
\end{equation*}

Evaluating \( \mathcal{W}_n z - \mathcal{W} z \), we have

\begin{equation*}
\begin{pmatrix}
\sum_{i=n+1}^\infty - b_i \sin(ix) \\
\sum_{i=n+1}^\infty -(-a_i i - \gamma c_i)\sin(ix) \\
\sum_{i=n+1}^\infty -(\gamma b_i - c_i i^2)\sin(ix)
\end{pmatrix} = R_n.
\label{eq4}
\end{equation*}

From the fact that \( \mathcal{W}z \in \mathcal{H} \), we have \( \| R_n \| \to 0 \) as \( n \to \infty \) and consequently, we get

\begin{equation*}
\lim_{n\to+\infty} \|\mathcal{W}_n z -  \mathcal{W}z \|_{\mathcal{H}} = 0, ~~ \text{ for all } z\in \mathcal{D}.
\label{eq5}
\end{equation*}
\end{proof}

\subsubsection{Approximation of $(S_w)$ under Neumann--Dirichlet (BCs)}

Under Neumann--Dirichlet \textbf{(BCs)},
\begin{equation*}
\partial_x u|_{x=0,\pi} = 0 \quad \text{and} \quad  \theta|_{x=0,\pi} = 0 \quad \text{for} \quad t > 0.
\end{equation*}
the basis vectors are given by
\[
\phi_{j}=\sqrt{\frac{2}{\pi}} \frac{1}{j} \cos j x, \quad \psi_{j}=\sqrt{\frac{2}{\pi}} \sin j x, \quad \xi_{j}=\sqrt{\frac{2}{\pi}} \sin j x, \quad j=1, \ldots, n.
\]
This leads to the following representation of the matrix $\mathcal{W}_{n}$
\[
\mathcal{W}_{n}=\left[\begin{array}{ccc}
0 & D_{n}^{T} & 0 \\
-D_{n} & 0 & -\gamma F_{n} \\
0 & \gamma F_{n}^{T} & -F^2_{n}
\end{array}\right].
\]
where \( D_{n} \) and \( F_{n} \) are defined as
$$
D_{n} = (D_{ij})_{1\leqslant i,j\leqslant n},\text{ with } D_{ij} = \begin{cases}-\frac{4}{\pi} \frac{i j}{i^{2}-j^{2}}, & |i-j| \text { is even } \\
0 & \text { otherwise }\end{cases}, 
$$
\[
F_n= \left[\begin{array}{lll}
1 & & \\
& \ddots & \\
& & n
\end{array}\right].
\]
Following the same proof as in Section \ref{2_sub_DN}, one can show the following result. 
\begin{prop}
\label{prop2}
\begin{itemize}
\item The semigroups generated by \( \mathcal{W}_{n} \) are uniformly polynomially stable with order $\alpha=2$.
\item The strong convergence of the approximate semigroups \( T_n(t) \) to \( T(t) = e^{t\mathcal{W}}, t>0, \) still holds.
\end{itemize}
\end{prop} 
\subsubsection{Approximation of $(S_w)$ under Neumann--Neumann (BCs)}

Under Neumann--Neumann \textbf{(BCs)}, 

\begin{equation*}
\partial_x u|_{x=0,\pi} = 0 \quad \text{and} \quad  \partial_x\theta|_{x=0,\pi} = 0 \quad \text{for} \quad t > 0,
\end{equation*}
the basis vectors are given by
\[
\phi_{j}=\sqrt{\frac{2}{\pi}} \frac{1}{j} \cos j x, \quad \psi_{j}=\sqrt{\frac{2}{\pi}} \sin j x, \quad \xi_{j}=\sqrt{\frac{2}{\pi}} \cos j x, \quad j=1, \ldots, n.
\]
This leads to the following representation of the matrix $\mathcal{W}_{n}$
\[
\mathcal{W}_{n}=\left[\begin{array}{ccc}
0 & D_{n}^{T} & 0 \\
-D_{n} & 0 & -\gamma F_{n} \\
0 & \gamma F_{n}^{T} & -G_{n}
\end{array}\right],
\]
where \( D_{n} \), \( F_{n} \) and \(G_{n} \) are defined as
$$
D_{n} = (D_{ij})_{1\leqslant i,j\leqslant n},\text{ with } D_{ij} = \begin{cases}-\frac{4}{\pi} \frac{i j}{i^{2}-j^{2}}, & |i-j| \text { is even } \\
0 & \text { otherwise }\end{cases}, 
$$
$$ F_n= (F_{ij})_{1\leqslant i,j\leqslant n},\text{ with } F_{ij} = \begin{cases}-\frac{4}{\pi} \frac{i}{j^{2}-i^{2}}, & |i-j|=\text { even } \\
0 & \text { otherwise }\end{cases}, 
$$
\[
G_n = \frac{2}{\pi} \left[\begin{array}{lll}
1 & & \\
& \ddots & \\
& & n^2
\end{array}\right].
\]
Following the same line of reasoning as in Section~\ref{2_sub_DN}, one can show the following result. 
\begin{prop}
\label{prop3}
\begin{itemize}
\item The semigroups generated by \( \mathcal{W}_{n} \) are uniformly polynomially stable with order $\alpha=2$.
\item The strong convergence of the approximate semigroups \( T_n(t) \) to \( T(t) = e^{t\mathcal{W}}, t>0, \) still holds.
\end{itemize}
\end{prop} 

\section{Numerical Studies}\label{section5}
In this section, we present numerical experiments that illustrate the asymptotic behavior of the strongly and weakly coupled thermoelastic systems discussed in the previous sections \ref{section3} and \ref{section4}. Our primary objective is to validate the theoretical results on exponential and polynomial decay by examining the evolution of the system energy under various boundary conditions and initial data profiles. In particular, we emphasize the influence of initial data regularity and the nature of coupling on the long-time behavior of solutions. The numerical simulations also serve to highlight the spectral characteristics observed in the modal approximation schemes and confirm the spectral asymptotics derived analytically in Section \ref{section4}. All computations are performed using modal approximation schemes to ensure that the spectral features are accurately captured.

\subsection{Spectral Behavior of Thermoelasticity}

Figures \ref{fig1} and \ref{fig2} illustrate, respectively, the impact of the coupling terms on the placement of eigenvalues of the dynamic matrix $\mathcal{S}_{n}$ and $\mathcal{W}_{n}$ in the case of the Dirichlet-Dirichlet \textbf{BCs}.

In Figure \ref{fig1}, the distance between the eigenvalues and the imaginary axis remains uniform, as shown in Table \ref{dis_spec}. This is in perfect agreement with Theorem \ref{thm3}. It is also observed that for a fixed $n$, the eigenvalues of higher frequency modes, particularly the $n^{th}$ mode, are closer to the imaginary axis. Additionally, as the number of modes increases, these eigenvalues revert to the vertical line $\Re\lambda=-\frac{\gamma^2}{2}$, a phenomenon previously demonstrated by \citet{Hansen1992,guo1997asymptotic}. Consequently, the corresponding spectral element approximation scheme preserves the property of exponential stability.

\begin{table}[h]
\centering
\caption{Distance between $\sigma(\mathcal{S}_{n})$ and the imaginary axis for the spectral element method}
\begin{tabular}{|c|c|} 
\hline
$n$ & $\min\{-\text{Re}(\lambda), \lambda\in\sigma(\mathcal{S}_{n})\}$\\
\hline 
8 & $8.9227\times 10^{-4}$\\
\hline
16 & $8.9383\times 10^{-4}$\\
\hline
24 & $8.9402\times 10^{-4}$\\
\hline
32 & $8.9407\times 10^{-4}$\\ 
\hline
\end{tabular}
\label{dis_spec}
\end{table}

\begin{figure}[h]
\centering
\includegraphics[scale=0.6]{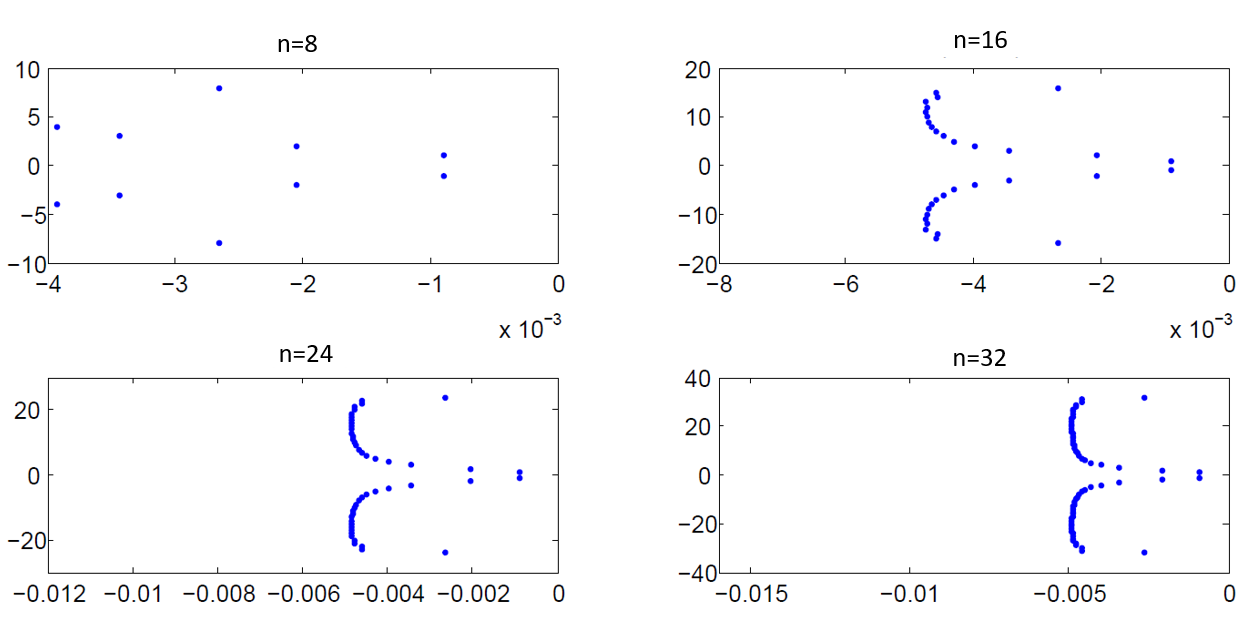}
\caption{Eigenvalues of $\mathcal{S}_{n}$.}
\label{fig1}
\end{figure}

\begin{figure}[H]
\centering
\includegraphics[scale=0.5]{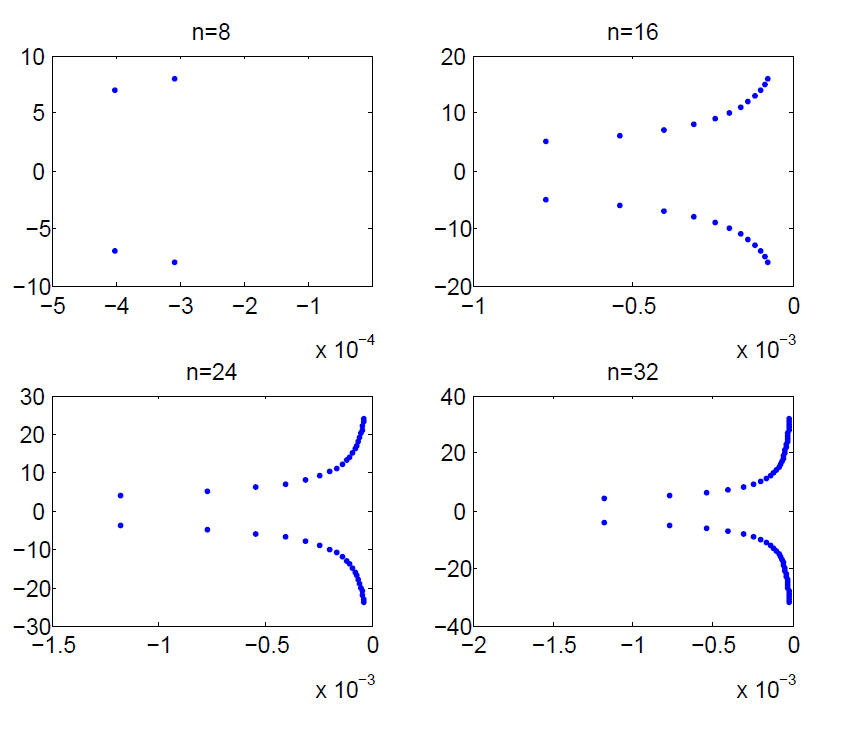}
\caption{Eigenvalues of $\mathcal{W}_{n}$.}
\label{fig2}
\end{figure}

In contrast to Figure \ref{fig1}, where the eigenvalues maintain a uniform distance from the imaginary axis, Figure \ref{fig2} reveals that as the number of modes increases, an asymptotic behavior emerges near the imaginary axis at $\pm\infty$. This property is a hallmark of systems with polynomial decay, as discussed in \cite{batkai2006polynomial}.

\subsection{Energy Behavior of Thermoelasticity}

The discrete energy associated to systems $(S_*)$ is given by

\begin{equation}
\label{eq4b}
E_{*,n}(t)=\frac{h}{2}\sum_{j=1}^n\Bigg\{\Bigg|\frac{u_{j+1}(t)-u_{j}(t)}{h}\Bigg|^{2}+|v_j(t)|^{2}+|\theta_j(t)|^{2}\Bigg\},\;*\in\{s,w\}.
\end{equation}

For system $(S_s)$, the discrete energy $E_{s,n}$ decays exponentially to zero (see Figure \ref{fig3}). Specifically, there exist positive constants $M$ and $\alpha$ such that
\[
E_{s,n}(t)\leqslant Me^{-\alpha t}E_{s,n}(0),\;n\in\mathbb{N},\;t>0.
\]
However, the introduction of the weak coupling term in system $(S_w)$ alters the dynamics and consequently the behavior of the energy \eqref{eq4b}. In this case, the system $(S_w)$ exhibits polynomial decay to zero (see Figure \ref{fig3}). Specifically, there exist positive constants $M$ such that
\[
E_{w,n}(t)\leqslant \frac{M}{t}\|\mathcal{W}_{n}z_{n0}\|^2,\;n\in\mathbb{N},\;t>0.
\]

\begin{figure}[H]
\centering
\includegraphics[scale=0.45]{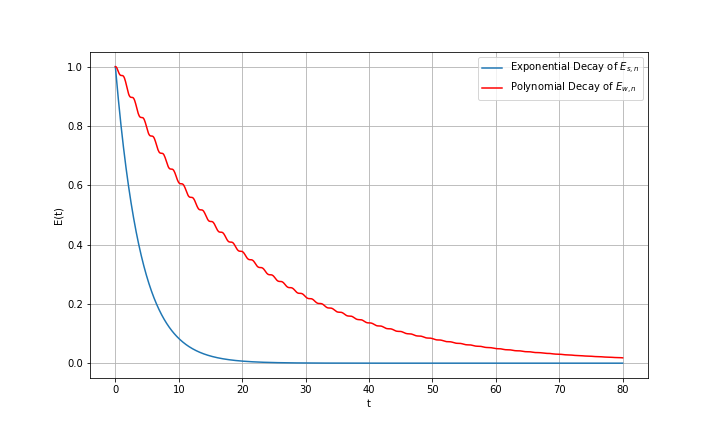}
\caption{Uniform decay of the energy.}
\label{fig3}
\end{figure}

\subsection{Effect of smoothness of the initial data on the rate of decay of energy} 

For the following numerical simulations, we use a uniform mesh with $n = 100$ elements, fix the final time at $T = 100$, set $dt = 0.1$, and consider the following initial conditions for $u$, $u_t$ and $\theta$

\[
u(x,0) = 0,\quad \theta(x, 0) = 0,\quad u_t(x, 0) =sin(jx),\; j = 1, 2, 3.
\]

\begin{figure}[H]
\centering
\includegraphics[scale=0.8]{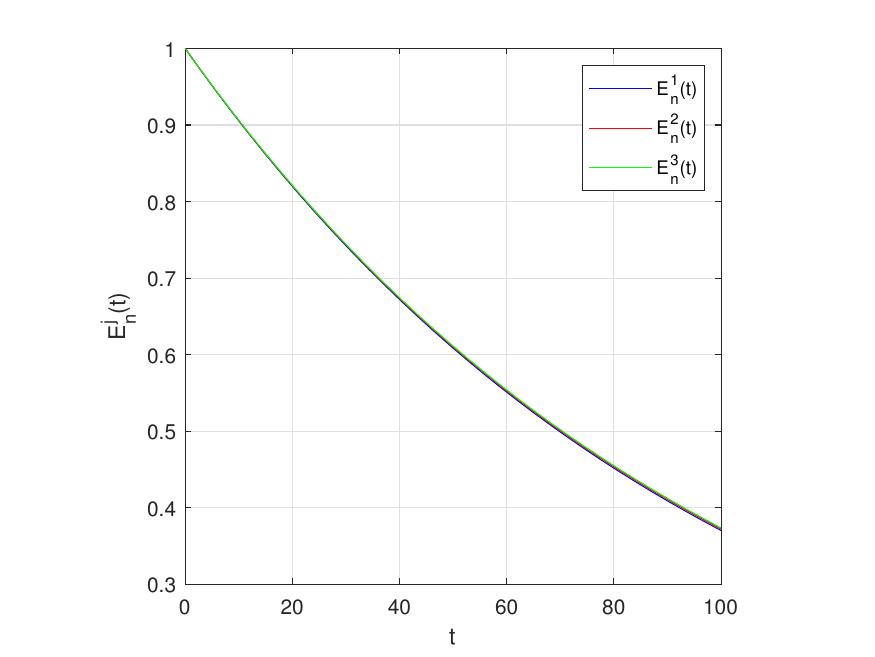}
\caption{No effect of initial data smoothness on the exponential decay of $E_{s,n}(t)$.}
\label{fig4}
\end{figure}

The behavior of the energy associated with system $(S_s)$ remains unaffected by the smoothness of the initial data as $n\to\infty$, as shown in Figure \ref{fig4}.

In contrast, it has been shown (see \cite{batkai2006polynomial, borichev2010optimal}) that the energy associated with system $(S_w)$ is highly sensitive to the smoothness of its initial data. This phenomenon is also witnessed in our numerical observations, as illustrated in Figure \ref{fig5}.

\begin{figure}[H]
\centering
\includegraphics[scale=0.8]{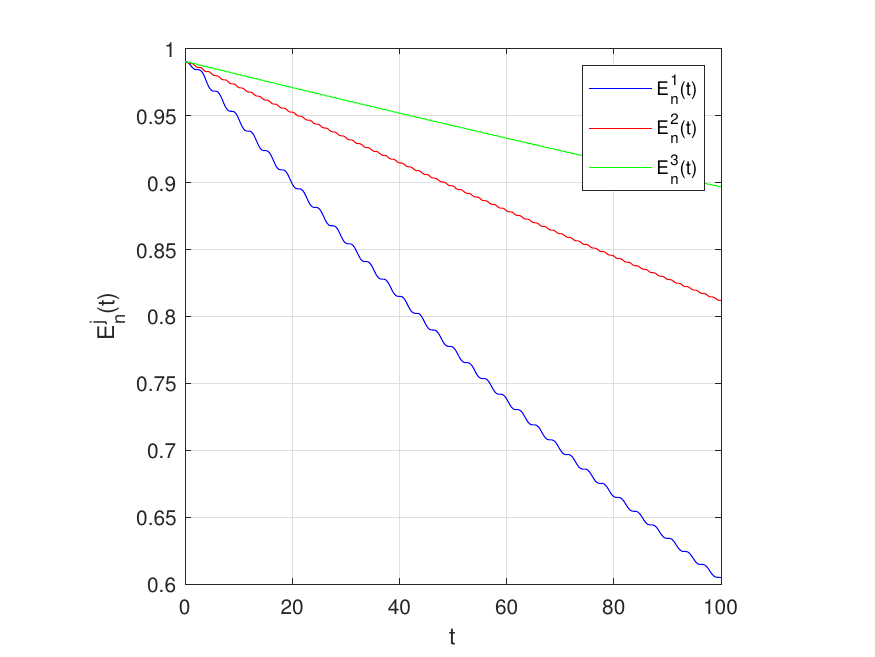}
\caption{Effect of initial data smoothness on the polynomial decay of $E_{w,n}(t)$.}
\label{fig5}
\end{figure}

In Figure \ref{fig5}, we see that for $j = 1$, the approximate energy $E_{w,n}(t)$ decays to zero as time $t$ increases. However, the decay rate is strongly influenced by $j$. Specifically, as $j$ increases, and thus the initial data become more oscillatory. This indicates that the rate of decay of the discrete energy $E_{w,n}(t)$ is highly sensitive to the choice of initial data.

\subsection{Effect of discontinuity of the initial data on the rate of decay of energy}

For this numerical simulations, we consider the following initial conditions for $u$, $u_t$ and $\theta$

\[
u(x,0) = 0,\quad \theta(x, 0) = 0,\quad u_t(x, 0) =2 \mathbf{1}_{[0, \pi]}(x) - 3 \mathbf{1}_{[\pi/2, \pi]}(x)
\]

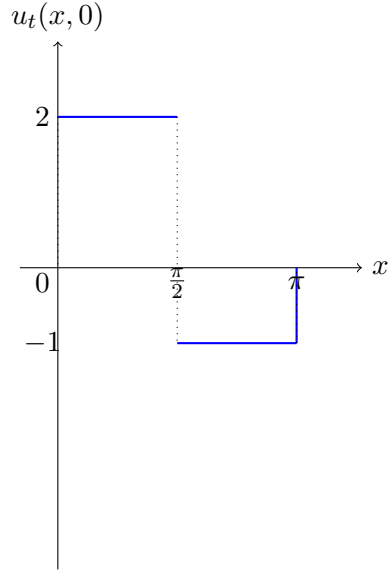
\begin{figure}[H]
\centering
\begin{tikzpicture}
    \draw[->] (-0.5, 0) -- (4, 0) node[right] {$x$};
    \draw[->] (0, -4) -- (0, 3) node[above] {$u_t(x, 0)$};

    \draw[thick, blue] (0, 2) -- (1.5708, 2);
    \draw[thick, blue] (1.5708, -1) -- (3.14159, -1);
    \draw[thick, blue] (3.14159, 0) -- (3.14159, -1);

    \node at (3.14159, -0.2) {$\pi$};
    \node at (1.5708, -0.2) {$\frac{\pi}{2}$};
    \node at (-0.2, 2) {$2$};
    \node at (-0.2, -1) {$-1$};
    \node at (-0.2, -0.2) {$0$};

    \draw[dotted] (0, 2) -- (0, 0);
    \draw[dotted] (1.5708, 2) -- (1.5708, -1);
    \draw[dotted] (3.14159, -1) -- (3.14159, 0);

\end{tikzpicture}
\caption{Plot of the discontinuous velocity $u_t(x,0)$.}
\label{fig6}
\end{figure}

From Figure \ref{fig7}, we observe that the energy associated with system $(S_s)$ remains unaffected by discontinuities in the initial data. However, the energy associated with system $(S_w)$, as shown in Figure \ref{fig7}, is sensitive to such discontinuities. This sensitivity indicates that the rate of decay of the discrete energy $E_{w,n}(t)$ in system $(S_w)$ is dependent on the choice of initial data.
 
\begin{figure}[H]
\centering
\includegraphics[scale=0.45]{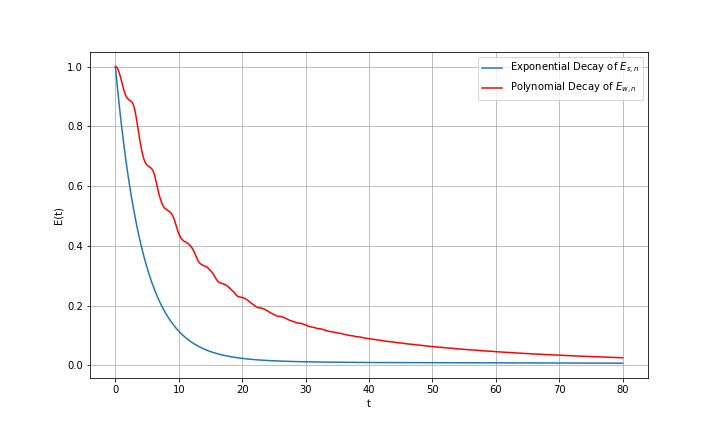}
\caption{Effect of initial data discontinuities on the polynomial decay of energy.}
\label{fig7}
\end{figure}



\section{Conclusion and Perspectives}
\label{section6}
In this study, we review and extend existing results on exponential and polynomial stability of coupled thermoelastic systems. Specifically, we establish that uniform decay can be achieved by employing a modal approximation framework combined with uniform resolvent estimates; see Theorems \ref{thm3} and \ref{thm4.4}. A key contribution of this work is the proof of uniform stability for coupled thermoelastic systems under a broad class of physically relevant boundary conditions.

The main contributions of the paper are as follows.
\begin{itemize}
\item We generalize exponential stability results from strongly to weakly coupled thermoelastic systems.
\item We derive uniform polynomial decay rates using resolvent-based spectral analysis.
\item We consider a wide class of boundary conditions \textbf{(BCs)} relevant to engineering and physical models.
\item We demonstrate that modal approximations preserve the qualitative stability properties of the continuous system, ensuring accurate and robust numerical behavior.
\end{itemize}
Although the proposed framework is effective for weakly coupled thermoelastic systems, several avenues remain open for future research. A primary direction is to extend the analysis to non-linear weakly coupled systems, particularly those with temperature-dependent or displacement-dependent coefficients. Further investigation is also warranted in the context of complex geometries or higher-dimensional domains, where boundary regularity and coupling effects pose significant analytical challenges. Additionally, the framework can be enriched to incorporate dynamic, non-local,, or feedback-type boundary conditions, which are common in control theory and smart material systems. Another promising direction involves the inclusion of additional physical effects, such as memory terms, electromagnetic interactions, viscoelastic behavior, or fluid-structure coupling, to develop a multiphysics generalization of the current model.

\section*{Acknowledgment}

This work is dedicated to the E-Tacsi Team, in recognition
of their continuous support and valuable contributions to the study of thermoelasticity.
The authors gratefully acknowledge the support of the ´Ecole Hassania
des Travaux Publics, in particular the MoNum Team, for providing a stimulating
research environment. They also thank Mrs. Ghabbar Yamna, Mr. Najib Cherfaoui
and Mr. Noureddine Semane for their valuable assistance, encouragement,
and sustained support throughout this work. Finally, the authors express their
sincere gratitude to the anonymous referees for their careful reviews and constructive comments, which have significantly improved the quality and clarity of this manuscript.

\section*{Statements and Declarations}

The authors of this paper declare that they have no known competing financial interests or personal relationships that could have appeared to influence the work reported in this paper. 

\bibliographystyle{plainnat}
\bibliography{biblio}
\end{document}